\documentclass[11pt]{article}
\usepackage{amssymb}
\usepackage{amsfonts}
\usepackage{amsmath}
\usepackage{mathrsfs}
\usepackage{graphicx}
\usepackage{amsbsy}
\usepackage{theorem}
\usepackage{color}
\usepackage{hyperref}
\usepackage{tikz}
\usepackage[normalem]{ulem}
\newtheorem{theorem}{Theorem}[section]
\newtheorem{lemma}[theorem]{Lemma}
\newtheorem{corollary}[theorem]{Corollary}
\newtheorem{definition}[theorem]{Definition}
\newtheorem{proposition}[theorem]{Proposition}
\newtheorem{remark}[theorem]{Remark}

\def\thetheorem{\thesection.\arabic{theorem}}
\def\thesection{\arabic{section}}

\def\theequation {\thesection.\arabic{equation}}
\def\beq{\begin{equation}\displaystyle}
\def\eeq{\end{equation}}
\def\bel{\begin{equation} \displaystyle \begin{array}{l} }
\def\eel{\end{array} \end{equation} }
\def\bell{\begin{equation} \displaystyle \begin{array}{ll}  }
\def\eell{\end{array} \end{equation} }

\def\bea{\begin{eqnarray}}
\def\eea{\end{eqnarray} }
\def\bean{\begin{eqnarray*}}
\def\eean{\end{eqnarray*} }
\newenvironment{proof}{\noindent{\bf Proof.~}}
{{\mbox{}\hfill {\small \fbox{}}\\}}
\def\qed{\mbox{}\hfill {\small \fbox{}}\\}
\catcode`@=11
\renewcommand\appendix{\bigskip {\noindent \Large \bf Appendix}
  \renewcommand{\thesection}{\Alph{section}}
  \setcounter{section}{0}%
  \setcounter{subsection}{0}%
\setcounter{equation}{0}%
\setcounter{theorem}{0}%
\def\thetheorem{A.\arabic{theorem}}
\def\theequation {A.\arabic{equation}}}
\catcode`@=12

\def\RR{\mathbb{R}}

\def\ds{\displaystyle}

\def\eps{\varepsilon}

\def\calW{{\mathcal W}}

\begin{document}

\title{On a reaction-diffusion system modeling strong competition between two mosquito populations}
\author{Nicolas Vauchelet\thanks{Universit\'e Sorbonne Paris Nord, Laboratoire Analyse, G\'eom\'etrie et Applications, LAGA, CNRS UMR 7539,
    F-93430, Villetaneuse, France (\texttt{vauchelet@math.univ-paris13.fr}) }}

\maketitle


\begin{abstract}
This paper is devoted to the analysis of a reaction-diffusion system with strong competition and spatial heterogeneities modelling the interaction between two species of mosquitoes.
In particular, we propose a mathematical model that accounts for the spatial segregation observed between two species of mosquito vectors of numerous viruses.
Indeed, it has been observed that, in tropical regions, \textit{Aedes aegypti} mosquitoes are well established in urban area whereas \textit{Aedes albopictus} mosquitoes spread widely in forest region. Moreover, these species of mosquitoes compete with each other in the larval stage.
Based on these observations, we introduce a simple mathematical model to account for this phenomenon.
This model consists of a system of reaction-diffusion equations describing the dynamics of the aquatic and aerial phases of each species in a spatially heterogeneous environment. The competition takes place at the aquatic phase and is assumed to be strong which allows us to reduce the dimensionality of the system.
We first establish a sufficient condition on the parameters to prevent one species from invading another in a homogeneous environment. Next, using this sufficient condition, we show that spatial segregation may be observed in a spatially heterogeneous environment. 
Our theoretical results are also illustrated by some numerical simulations.
\end{abstract}
\vskip1cm


\noindent{\textbf{\small AMS Subject Classification:}  35B25, 35B27, 35B45, 35B51, 54C70, 82B21.

  \section{Introduction}

  \textit{Aedes} mosquitoes are a major vector of several diseases, among them dengue, yellow fever, chikungunya, Zika. Originally from tropical regions in Asia, they have widely spread and are now well-established in many countries, including temperate regions (see e.g. \cite{ECDC, USA, SA, Africa}).
  Understanding the dynamics of these species of mosquitoes is fundamental in the fight against the diseases they transmit. 
  Among \textit{Aedes} mosquitoes, the most common are \textit{Aedes albopictus} and \textit{Aedes aegypti}. Inter-specific competition exists between \textit{Aedes aegypti} and \textit{Aedes albopictus} above all in the larval stage \cite{Competition1}.
Another interesting observation for these species is that their geographical distribution may be quite different~: \textit{Aedes aegypti} are mainly present in urban area such as city center whereas \textit{Aedes albopictus} are most widespread in suburb area or in forested area, see e.g. \cite{GeoDistrib, Congo}.

The aim of this paper is to study a mathematical model of reaction-diffusion equations that accounts for these observations.
Standard reaction-diffusion equations are often used to describe biological invasions in mathematical biology and ecology, see e.g. \cite{Murray, Cantrell, Shigesada, Zhao, Perthame, Roques}.
Recently two different approaches have been proposed to investigate the competition between \textit{Aedes aegypti} and \textit{Aedes albopictus}. In \cite{Competition}, the authors introduce an advection-reaction-diffusion system with free boundaries to study the invasion of these two species.
In \cite{samia}, a mathematical model, for which segregation between the two species with respect to their habitats occurs, is proposed. More precisely, this model consists in a scalar reaction-diffusion equation in a spatially heterogeneous environment. Explicit conditions on the parameters to guarantee invasion of one species or another are provided. This scalar equation is derived by reducing a competitive system of two reaction-diffusion equations by assuming that competition is large, such a reduction follows the ideas developed in \cite{Dancer1, Dancer2, Girardin}.
However in \cite{samia} the dynamics of the aquatic phase is neglected to simplify the analysis.

In this article, we extend this latter model by incorporating the dynamics of the aquatic phase. Then, we divide each population into an aquatic stage and an adult stage. This, of course, increases the dimensionality of the problem and thus the difficulty of the mathematical analysis.
So, although in \cite{samia} the system is reduced to a scalar reaction-diffusion equation, here the resulting model is a competitive system of reaction-diffusion equations.
Although the existence of traveling front for competitive system is now well-established (see \cite{Gardner}), finding conditions to guarantee invasion of one species or another is still an active area of research, see e.g. the review article \cite{GirardinReview}.
Then, a first objective of this work is to provide an explicit sufficient condition to guarantee invasion of one species in the resulting competitive system.
Moreover, in order to model two different habitats (urban area and forest), we consider a two-patch landscape by assuming that the carrying capacities are different in the two patches, and we investigate the biological invasion in the whole domain.
Mathematical analysis of invasion phenomena in reaction-diffusion equations in spatially heterogeneous domain has been investigated in several works, for instance in periodic environments we refer to the articles \cite{SKT, BHN, Nadin, BerestyckiNadin} and references therein, we refer also to \cite{Maciel, Ovaskainen, Hamel} where patchy models with interface matching conditions are considered.
In the main result of this article, we provide some explicit conditions on the parameters to guarantee invasion of each species in its habitat under some non-smallness condition on the initial data.

The outline of this paper is the following. In Section \ref{sec:mainresult}, the mathematical models are introduced and the main results are summarized. 
Section \ref{sec:full} is devoted to a brief analysis of the equilibria and their stability for the full system and its reduction in the strong competition asymptotic, whose rigorous proof is postponed in an appendix at the end of the paper. Section \ref{sec:homogene} deals with the mathematical analysis of the reduced system in a spatially homogeneous environment. In particular we provide a sufficient condition on the parameters to guarantee invasion of one species.
This result is illustrated with a numerical simulation.
Then, in Section \ref{sec:heterogene} we investigate the spatially heterogeneous case by considering a two-patch environment. We provide some sufficient conditions to have invasion of each species in its habitat. This result is also illustrated with a numerical simulation.
This articles ends with a concluding Section with outlooks and perspectives. Finally, an Appendix gathers some useful technical proofs.

\section{Mathematical modelling and main results}\label{sec:mainresult}

The starting point of this study is the following competitive model in one space dimension, which is inspired by \cite{SBD, 2patch}.
Let $E_1$, resp. $E_2$, denote the density of species $1$, resp. species $2$, at aquatic phase, and let $F_1$, resp. $F_2$, be the density of fertilized female adults of species $1$, resp. species $2$.
We assume that these quantities depends on the time $t>0$ and on the space variable $x\in \RR$.
The model reads
\begin{subequations}\label{syst}
\begin{align}\label{syst:E1}
&\partial_t E_1 = b_1 F_1 \left(1 - \frac{E_1}{K_1}\right) - c E_1 E_2 - (\mu_1+\nu_1) E_1 \\
\label{syst:F1}
&\partial_t F_1 - D_1 \partial_{xx} F_1 = \rho \nu_1 E_1 - \delta_1 F_1  \\
\label{syst:E2}
&\partial_t E_2 = b_2 F_2 \left(1 - \frac{E_2}{K_2}\right) - c E_1 E_2 - (\mu_2+\nu_2) E_2 \\
\label{syst:F2}
&\partial_t F_2 - D_2 \partial_{xx} F_2 = \rho \nu_2 E_2 - \delta_2 F_2.
\end{align}
\end{subequations}
In this system, fertilized females give birth to mosquitoes at aquatic phase with a birth rate denoted $b_1$ for species $1$ and $b_2$ for species $2$. The death rates of mosquitoes at aquatic phases are denoted $\mu_1$ for species $1$ and respectively $\mu_2$ for species $2$. For adult females, the death rates are $\delta_1$ and $\delta_2$. The transition from the aquatic phase to the adult phase is called the emergence. The emergence rates are denoted $\nu_1$ and $\nu_2$ and the sex ratio, corresponding to the probability that a female emerges, is $\rho$. The competition parameter is $c$, competition occurs only at aquatic phase. All these quantities are positive constants. Finally, the growth is logistic and we consider the carrying capacities, $K_1$ and $K_2$, which are positive functions which may depend on the space variable and we assume for $i=1, 2$,
$$
0< K_i \in L^\infty (\RR), \qquad 0 < \frac{1}{K_i} \in L^\infty(\RR).
$$
Finally, this system is complemented with non-negative initial conditions
$$
E_1(t=0) = E_1^0,\quad F_1(t=0)= F_1^0,\quad E_2(t=0) = E_2^0,\quad F_2(t=0)=F_2^0.
$$
System \eqref{syst} is competitive. Neglecting spatial diffusion and spatial dependency, we first show in Proposition \ref{prop:equilib} below that, under the natural assumption that the basic reproduction number of each species is greater than 1, the only stable equilibria are the ones with only one species, i.e. the two stable equilibria are $(E_1^*,F_1^*,0,0)$ and $(0,0,E_2^*,F_2^*)$, where $E_i^* = K_i(1-\frac{\delta_i(\mu_i+\nu_i)}{b_i\rho\nu_i})$ and $F_i^* = \frac{\rho\nu_i}{\delta_i} E_i^*$, $i=1,2$.

In this paper we are interested in the case of a strong competition, that is $c\gg 1$. Formally, if we let $c$ going to $+\infty$ in system \eqref{syst}, we deduce that, at the limit 
$$
E_1 E_2 = 0.
$$
In other words, the two species are segregated at the aquatic stage. Then, we introduce the quantity $w=E_1-E_2$ such that $E_1 = w_+$ and $E_2=w_-$ where $w_+$ (resp. $w_-$) denotes the positive (resp. negative) part of $w$, i.e. $w_+ = \max\{w,0\}$ and $w_- = \max\{-w,0\}$.
From \eqref{syst:E1} and \eqref{syst:E2}, we deduce easily that $w$ solves the following equation
\begin{equation}\label{eq:w}
\partial_t w = b_1 F_1 \left(1 - \frac{w_+}{K_1}\right) - (\mu_1+\nu_1) w_+ - b_2 F_2 \left(1 - \frac{w_-}{K_2}\right) + (\mu_2+\nu_2) w_-,
\end{equation}
coupled with \eqref{syst:F1} and \eqref{syst:F2} which reads now
\begin{subequations}\label{systF}
\begin{align}\label{systF:1}
&\partial_t F_1 - D_1 \partial_{xx} F_1 = \rho \nu_1 w_+ - \delta_1 F_1  \\
\label{systF:2}
&\partial_t F_2 - D_2 \partial_{xx} F_2 = \rho \nu_2 w_- - \delta_2 F_2.
\end{align}
\end{subequations}
Therefore, the asymptotic of strong competition allows us to reduce system \eqref{syst} of four equations to system \eqref{eq:w}-\eqref{systF} with three equations.
The rigorous proof of this reduction follows the arguments in \cite{samia} and are briefly recalled and adapted for this framework in the appendix of this paper (see Theorem \ref{th:reduc}). 

The mathematical model under investigation in this paper is then the reduced system \eqref{eq:w}--\eqref{systF}. For this system, we study the large time behavior in a homogeneous domain in Section \ref{sec:invhom} and in a spatially heterogeneous domain in Section \ref{sec:heterogene}.
One may summarize our main results as the following~: Assuming that the conditions $b_i \rho \nu_i > \delta_i (\mu_i+\nu_i)$ for $i=1,2$ hold.
\begin{itemize}
\item In a spatially homogeneous environment, i.e. when all parameters are constants~: Under some analytical condition on the parameters (see \eqref{eq:Gamma}--\eqref{hypGamma}), if the initial quantity of species $1$, $F_1^0$, is large enough, then the solution of \eqref{eq:w}--\eqref{systF} converges as $t$ goes to $+\infty$ towards the equilibria where only species $1$ remains, i.e. $F_1(t,x) \to F_1^*:=\frac{K_1}{b_1 \delta_1}(\rho\nu_1 b_1 - \delta_1 (\mu_1+\nu_1)) >0$ and $F_2(t,x) \to 0$ for all $x\in\RR$ (Theorem \ref{th:invasion1}).
\item In a two-patch landscape, i.e. when $K_i(x) = K_i^F \mathbf{1}_{\{x<0\}} + K_i^U \mathbf{1}_{\{x\geq 0\}}$, for $i=1,2$~: Under some analytical conditions on the parameters, we assume that the initial data are such that~: The density of species $1$ is initially large enough in the domain $\{x<0\}$ and the density of species $2$ is large enough in the domain $\{x>0\}$. Then
  the solution of \eqref{eq:w}--\eqref{systF} is such that as time goes to $+\infty$ there is invasion of the species $1$ in the domain $\{x<0\}$ and invasion of the species $2$ in the domain $\{x>0\}$.
\end{itemize}
System \eqref{eq:w}--\eqref{systF} being competitive, a comparison principle allows us to compare solutions (see Proposition \ref{prop:compar}). 
Then, in order to establish such results of invasion, we construct sub-solutions and super-solutions of the stationary system. The main difficulty is to find a condition to guarantee invasion, to do so we construct a stationary solution in the half space $\RR^+$ which links the two constant equilibria. Then, we use this stationary solution to build a sub-solution and a super-solution in the whole space $\RR$.

To conclude this section, we emphasize that in the statement of our results, it is necessary to assume that the initial data of the invading species is large enough to guarantee invasion. Indeed, due to the large competition, it seems natural that the initial density of the population should be large enough to succeed in invading. We will refer to such assumption as a \textit{non-smallness} assumption.
We mention that non-smallness assumptions to guarantee invasion of a solution of a scalar bistable reaction-diffusion equations are well-known. Indeed, in the seminal work \cite{Zlatos}, the author proves that for a scalar bistable reaction-diffusion equation with initial data $\mathbf{1}_{[-L,L]}$, there exists a critical size $L^*$ such that for $L<L^*$ the solution goes extinct whereas for $L>L^*$ there is invasion. This results has been extended in particular to more general initial data in several works, see e.g. \cite{Matano,  Muratov}.
The question of finding an optimal initial data to guarantee invasion has been also investigated in \cite{Ana, Ana2}.
In application to the control of \textit{Aedes} mosquitoes, the authors in \cite{BartonTurelli} construct initial data for which invasion occurs by considering a family of stationary solutions.
This idea has also been used in \cite{Zubelli, Bliman} in the framework of a population replacement.

\section{Reduction of the full system}\label{sec:full}

\subsection{Equilibria of the associated dynamical system}

Before considering the full reaction-diffusion system with spatial dependency, we first investigate the corresponding ODE system and study its equilibria and their stability when the carrying capacities are assumed to be constant. More precisely, let us consider the dynamical system
\begin{subequations}\label{systODE}
\begin{align}\label{systODE:E1}
&\frac{dE_1}{dt} = b_1 F_1 \left(1 - \frac{E_1}{K_1}\right) - c E_1 E_2 - (\mu_1+\nu_1) E_1 \\
\label{systODE:F1}
&\frac{dF_1}{dt} = \rho \nu_1 E_1 - \delta_1 F_1  \\
\label{systODE:E2}
&\frac{dE_2}{dt} = b_2 F_2 \left(1 - \frac{E_2}{K_2}\right) - c E_1 E_2 - (\mu_2+\nu_2) E_2 \\
\label{systODE:F2}
&\frac{dF_2}{dt} = \rho \nu_2 E_2 - \delta_2 F_2.
\end{align}
\end{subequations}
As above, this differential system is complemented with non-negative initial data.
Clearly, thanks to the Cauchy-Lipschitz theorem, there exists a global in time solution to this system, moreover the set $[0,K_1]\times \RR^+ \times [0,K_2]\times \RR^+$ is invariant.
The basic reproduction number for each species is given by
\begin{equation}\label{def:calN}
\mathcal{N}_i = \frac{b_i \rho \nu_i}{\delta_i (\mu_i+\nu_i)}, \qquad i\in \{1,2\}.
\end{equation}
This basic reproduction number is defined as the number of new individuals generated by one individual in the population. Hence, in the absence of the other species, each species is viable if its basic reproduction number is greater than $1$.
Then, we will always assume that
\begin{equation}\label{hypN}
\mathcal{N}_i > 1, \qquad i\in\{1,2\}.
\end{equation}
The following result provides the equilibria for system \eqref{systODE} and their stability.
\begin{proposition}\label{prop:equilib}
  Let us assume that \eqref{hypN} holds and introduce the positive quantities, for $i=1,2$,
  \begin{equation}\label{eq:E*F*}
    F_i^* = \frac{\rho\nu_i K_i}{\delta_i}\left(1-\dfrac{1}{\mathcal{N}_i}\right), \qquad
    E_i^* = K_i\left(1-\dfrac{1}{\mathcal{N}_i}\right).
  \end{equation}
  Then, system \eqref{systODE} has four non-negative equilibria~: the extinction equilibrium $(0,0,0,0)$, two equilibria with only one population $(E_1^*,F_1^*,0,0)$ and $(0,0,E_2^*,F_2^*)$, and a coexistence equilibrium $(\widetilde{E_1},\widetilde{F_1},\widetilde{E_2},\widetilde{F_2})$ which is non-negative provided
  \begin{equation}\label{condclarge}
  c > \max\left\{\frac{\rho\nu_2 b_2}{\delta_2 K_1}\times\frac{1-\frac{1}{\mathcal{N}_2}}{1-\frac{1}{\mathcal{N}_1}} ,\  \frac{\rho\nu_1 b_1}{\delta_1 K_2}\times\frac{1-\frac{1}{\mathcal{N}_1}}{1-\frac{1}{\mathcal{N}_2}}\right\}.
  \end{equation}

  Moreover, under condition \eqref{condclarge}, it holds~:
  \begin{itemize}
  \item[(i)] The extinction and the coexistence equilibria are unstable;
  \item[(ii)] The equilibria with one population $(E_1^*,F_1^*,0,0)$ and $(0,0,E_2^*,F_2^*)$ are locally asymptotically stable.
  \end{itemize}
\end{proposition}
\begin{proof}
  To compute the equilibria we look for constants solutions of system \eqref{systODE}, then solve the system obtain by replacing the left hand side of \eqref{systODE} by $0$.
  We get $\widetilde{E_i} = \frac{\delta_i}{\rho \nu_i} \widetilde{F_i}$ and
  \begin{align*}
    &  b_1 \left(1-\frac{\delta_1}{\rho\nu_1 K_1} \widetilde{F_1}\right) - c \frac{\delta_1 \delta_2}{\rho^2 \nu_1\nu_2} \widetilde{F_2}  - \frac{(\mu_1+\nu_1)\delta_1}{\rho\nu_1} = 0 , \quad \text{ or } \widetilde{F_1} = 0,   \\
    &  b_2 \left(1-\frac{\delta_2}{\rho\nu_2 K_2} \widetilde{F_2}\right) - c \frac{\delta_1 \delta_2}{\rho^2 \nu_1\nu_2} \widetilde{F_1}  - \frac{(\mu_2+\nu_2)\delta_2}{\rho\nu_2} = 0 , \quad \text{ or } \widetilde{F_2} = 0.
  \end{align*}
  We deduce easily that $(0,0,0,0)$, $(E_1^*,F_1^*,0,0)$, and $(0,0,E_2^*,F_2^*)$ are constants solutions. For the last case, after some straightforward computations, we obtain
  \begin{align*}
    & \widetilde{F_1} = \frac{\rho \nu_1 K_1}{\delta_1}\times \dfrac{1-\dfrac{1}{\mathcal{N}_1}-\dfrac{c\delta_1}{\rho\nu_1b_1} K_2 \left(1-\dfrac{1}{\mathcal{N}_2}\right)}{1 - \dfrac{c^2 \delta_1 \delta_2 K_1 K_2}{\rho^2 \nu_1 \nu_2 b_1 b_2}},    \\
    & \widetilde{F_2} = \frac{\rho \nu_2 K_2}{\delta_2}\times \dfrac{1-\dfrac{1}{\mathcal{N}_2}-\dfrac{c\delta_2}{\rho\nu_2b_2} K_1 \left(1-\dfrac{1}{\mathcal{N}_1}\right)}{1 - \dfrac{c^2 \delta_1 \delta_2 K_1 K_2}{\rho^2 \nu_1 \nu_2 b_1 b_2}},
  \end{align*}
  which are positive provided the competition parameter $c$ is large enough as in \eqref{condclarge}.

  To study the stability of the equilibria, we consider the Jacobian matrix of system \eqref{systODE}
  $$
  J(E_1,F_1,E_2,F_2) =
  \begin{pmatrix}
    -\frac{b_2F_1}{K_1} - cE_2 -\mu_1-\nu_1 & b_1(1-\frac{E_1}{K_1}) & -cE_1 & 0  \\
    \rho \nu_1 & -\delta_1 & 0 & 0 \\
    -c E_2 & 0 & -\frac{b_2 F_2}{K_2} - c E_1 - \mu_2 - \nu_2 & b_2 (1-\frac{E_2}{K_2}) \\
    0 & 0 & \rho\nu_2 & -\delta_2 
  \end{pmatrix}.
  $$

  It is obvious to verify that for the extinction equilibrium, the Jacobian matrix $J(0,0,0,0)$ has positive eigenvalues, which implies that this equilibrium is unstable.

  Let us consider the equilibrium $(E_1^*,F_1^*,0,0)$. The characteristic polynomial of the Jacobian matrix at this equilibrium is
  \begin{align*}
    & \det(J(E_1^*,F_1^*,0,0)-XI_4) \\
    & = \begin{vmatrix} -X-\frac{b_1F_1^*}{K_1}-\mu_1-\nu_1 & b_1(1-\frac{E_1^*}{K_1}) \\
  \rho\nu_1 & -X-\delta_1 \end{vmatrix}\times \begin{vmatrix} -X-\mu_2-\nu_2-cE_1^* & b_2 \\
  \rho\nu_2 & -X-\delta_2 \end{vmatrix} \\
    & = (X^2+X(\delta_1+\frac{b_1 F_1^*}{K_1} + \mu_1+\nu_1)+\frac{\delta_1 b_1 F_1^*}{K_1}) \\
    & \quad \times  (X^2 + X(\delta_2+\mu_2+\nu_2+cE_1^*) + \delta_2(\mu_2+\nu_2+cE_1^*)-\rho\nu_2 b_2)
  \end{align*}
  All roots of this polynomial have negative real parts if and only if $\delta_2(\mu_2+\nu_2+cE_1^*) > \rho\nu_2 b_2$, which is equivalent to
  $$
  c>\frac{\rho\nu_2 b_2 - \delta_2(\mu_2+\nu_2)}{\delta_2 E_1^*}
  = \frac{\rho\nu_2 b_2}{\delta_2 K_1}\dfrac{1-\dfrac{1}{\mathcal{N}_2}}{1-\dfrac{1}{\mathcal{N}_1}}.
  $$
  The computation for the equilibrium $(0,0,E_2^*,F^*_2)$ is similar.
  Hence, under the condition \eqref{condclarge}, the equilibria $(E_1^*,F_1^*,0,0)$ and $(0,0,E_2^*,F^*_2)$ are locally asymptotically stable.

  Finally, for the coexistence equilibrium, we notice that
  $$
  \det(J(\widetilde{E_1},\widetilde{F_1},\widetilde{E_2},\widetilde{F_2})) = \left(\frac{\rho^2 \nu_1 \nu_2 b_1 b_2}{K_1 K_2} - c^2 \delta_1 \delta_2 \right) \widetilde{E_1}\widetilde{E_2} < 0,
  $$
  where we use condition \eqref{condclarge} for the last inequality.
  Then, we deduce that the matrix $J(\widetilde{E_1},\widetilde{F_1},\widetilde{E_2},\widetilde{F_2})$ admits at least one positive eigenvalue, which implies the instability of this equilibrium.
\end{proof}

As a consequence of this result, when the competition is large enough, the only stable equilibria are the one with only one population. In the rest of the paper, we will introduce space dependency by assuming an active motion of the adult population.

\subsection{Strong competition asymptotic}

We summarize briefly some standard properties of solutions of system \eqref{syst}. 
Existence and uniqueness of a solution to system \eqref{syst} is standard, see e.g. \cite[\S 3]{Perthame}. Then, assuming that initial data are given in $L^1\cap L^\infty(\RR)$, there exists a unique weak solution $(E_1,F_1,E_2,F_2) \in (C_{\text{loc}}(\RR^+; L^1\cap L^\infty(\RR)))^4$.
Moreover, by parabolic regularity, we deduce that for $i=1,2$, $F_i\in C_{\text{loc}}(\RR^+; C^1(\RR))$, $E_i \in C_{\text{loc}}(\RR^+; C(\RR))$ and $\partial_t E_i \in C_{\text{loc}}(\RR^+; L^1\cap L^\infty(\RR))$.
We have also the following estimates~:
\begin{lemma}
  Let us assume that for $i=1,2$, we have $0\leq E_i^0 \leq \|K_i\|_\infty$ and $0\leq F_i^0 \in L^\infty(\RR)$. Then for all $t>0$, the solution of \eqref{syst} with initial data $(E_1^0,F_1^0,E_2^0,F_2^0)$ verifies for all $t>0$, $x\in\RR$,
  $$
  0 \leq E_i(t,x) \leq \|K_i\|_\infty, \qquad 0 \leq F_i(t,x) \leq \max\left\{\|F_i^0\|_\infty,\frac{\rho\nu_i \|K_i\|_\infty}{\delta_i}\right\}.
  $$
\end{lemma}

Finally, we postpone in the appendix the proof of the reduction of system \eqref{syst} towards \eqref{eq:w}-\eqref{systF}, see Theorem \ref{th:reduc}.

\section{Analysis of the reduced system}\label{sec:homogene}

This part is devoted to the analysis of the reduced system on $\RR$~:
\begin{equation}\label{reduced}
  \left\{
    \begin{array}{l}
      \partial_t w = b_1 F_1 \left(1-\dfrac{w_+}{K_1}\right) - (\mu_1+\nu_1) w_+ - b_2 F_2 \left(1-\dfrac{w_-}{K_2}\right) + (\mu_2+\nu_2) w_-   \\[3mm]
      \partial_t F_1 - D_1 \partial_{xx} F_1 = \rho \nu_1 w_+ - \delta_1 F_1   \\
      \partial_t F_2 - D_2 \partial_{xx} F_2 = \rho \nu_2 w_- - \delta_2 F_2,
    \end{array}
  \right.
\end{equation}
complemented with initial data $w(t=0,x) = w^0(x)$, $F_1(t=0,x) = F_1^0(x)$, $F_2(t=0,x) = F_2^0(x)$.

In order to study the long time dynamics, we will consider the stationary system.
The stationary equation of the first equation in \eqref{reduced} provides the following explicit expression of $w$ with respect to $F_1$ and $F_2$~:
\begin{equation}\label{def:calW}
w(x) = \calW(F_1,F_2)(x) := \frac{ (b_1F_1(x) -b_2 F_2(x))_+}{\frac{b_1}{K_1(x)} F_1(x) +\mu_1+\nu_1} - \frac{(b_2F_2(x) -b_1 F_1(x))_+}{\frac{b_2}{K_2(x)} F_2(x) +\mu_2+\nu_2}.
\end{equation}
Then, the stationary problem reduces to the following system~:
\begin{subequations}\label{statF}
\begin{align}\label{statF:1}
&- D_1 F_1'' = \frac{\rho \nu_1}{\frac{b_1}{K_1} F_1 +\mu_1+\nu_1} (b_1F_1 -b_2 F_2)_+ - \delta_1 F_1  \\
\label{statF:2}
  &- D_2 F_2'' = \frac{\rho \nu_2}{\frac{b_2}{K_2} F_2 +\mu_2+\nu_2} (b_2F_2 -b_1 F_1)_+ - \delta_2 F_2.
\end{align}
\end{subequations}
We first start with a comparison principle for system \eqref{reduced}.

\subsection{Comparison principle}

\begin{lemma}\label{lem:invariant}
  The set $[-K_2,K_1]\times\RR^+ \times \RR^+$ is invariant, i.e. if for all $x\in \RR$, we have $-K_2(x) \leq w^0(x) \leq K_1(x)$, $F_1^0(x) \geq 0$, $F_2^0(x) \geq 0$, then for all $t>0$ and $x\in\RR$ we have $-K_2(x) \leq w(t,x) \leq K_1(x)$, $F_1(t,x) \geq 0$, $F_2(t,x)\geq 0$.
\end{lemma}
\begin{proof}
  If there exists $\tau>0$ and $x_0\in\RR$ such that $w(\tau,x_0)>K_1(x_0)$, then $\partial_t w(\tau,x_0) < 0$. Hence if $w(t=0,x_0) \leq K_1(x_0)$ then for any $t>0$ we have $w(t,x_0)\leq K_1(x_0)$. Similarly we have the inequality $w(t,x_0)\geq -K_2(x_0)$. The non-negativity of $F_1$ and $F_2$ is a consequence of the maximum principle since we have $\partial_t F_i - D_i \partial_{xx} F_i + \delta_i F_i \geq 0$ from \eqref{reduced}.
\end{proof}

In this invariant set, system \eqref{reduced} is clearly a competitive system. Then, we may define the following notion of sub-solution and super-solution~:
\begin{definition}\label{def:sub-super}
  We say that $(\underline{w},\underline{F_1},\underline{F_2})$ and $(\overline{w},\overline{F_1},\overline{F_2})$ is a couple of sub-solution and super-solution of system \eqref{reduced} if
  for all $t\geq 0$ and $x\in\RR$, $(\underline{w}(t,x),\underline{F_1}(t,x),\underline{F_2}(t,x))\in [-K_2(x),K_1(x)]\times\RR^+\times \RR^+$, $(\overline{w}(t,x),\overline{F_1}(t,x),\overline{F_2}(t,x))\in [-K_2(x),K_1(x)]\times\RR^+\times \RR^+$, and if it satisfies
  \begin{equation}\label{def:sub}
    \left\{
      \begin{array}{l}
        \partial_t \underline{w} \leq b_1 \underline{F_1} \left(1-\dfrac{\underline{w}_+}{K_1}\right) - (\mu_1+\nu_1) \underline{w}_+ - b_2 \overline{F_2} \left(1-\dfrac{\underline{w}_-}{K_2}\right) + (\mu_2+\nu_2) \underline{w}_-  \\[3mm]
        \partial_t \underline{F_1} - D_1 \partial_{xx} \underline{F_1} \leq \rho \nu_1 \underline{w}_+ - \delta_1 \underline{F_1}   \\[2mm]
        \partial_t \overline{F_2} - D_2 \partial_{xx} \overline{F_2} \geq \rho \nu_2 \underline{w}_- - \delta_2 \overline{F_2},
      \end{array}
    \right.
  \end{equation}
  \begin{equation}\label{def:super}
    \left\{
      \begin{array}{l}
        \partial_t \overline{w} \geq b_1 \overline{F_1} \left(1-\dfrac{\overline{w}_+}{K_1}\right) - (\mu_1+\nu_1) \overline{w}_+ - b_2 \underline{F_2} \left(1-\dfrac{\overline{w}_-}{K_2}\right) + (\mu_2+\nu_2) \overline{w}_-  \\[3mm]
        \partial_t \overline{F_1} - D_1 \partial_{xx} \overline{F_1} \geq \rho \nu_1 \overline{w}_+ - \delta_1 \overline{F_1}   \\[2mm]
        \partial_t \underline{F_2} - D_2 \partial_{xx} \underline{F_2} \leq \rho \nu_2 \overline{w}_- - \delta_2 \underline{F_2},
        \end{array}
    \right.
  \end{equation}
  together with the following conditions on the initial data
  $$
  \underline{w}(t=0) \leq w^0 \leq \overline{w}(t=0), \quad
  \underline{F_1}(t=0) \leq F_1^0  \leq \overline{F_1}(t=0), \quad
  \underline{F_2}(t=0) \leq F_2^0  \leq \overline{F_2}(t=0).
  $$
\end{definition}
In the following Proposition we state a comparison principle, whose proof is postponed to the Appendix~:
\begin{proposition}\label{prop:compar}
  Let $(\underline{w}, \underline{F_1}, \underline{F_2})$ and $(\overline{w}, \overline{F_1}, \overline{F_2})$ be respectively sub- and super-solutions of \eqref{reduced} in the sense of Definition \ref{def:sub-super}.
  Let us assume moreover that $(\underline{w}, \underline{F_1}, \underline{F_2})\in [-K_2,K_1]\times\RR^+\times \RR^+$ and $(\overline{w}, \overline{F_1}, \overline{F_2})\in [-K_2,K_1]\times\RR^+\times \RR^+$.
  Then, for all $t>0$, we have
  $$
  \underline{w} \leq \overline{w}, \quad \underline{F_1} \leq \overline{F_1}, \quad \underline{F_2} \leq \overline{F_2}.
  $$
\end{proposition}

A direct consequence of this result is the following estimate~:
\begin{corollary}\label{cor:boundF}
  Let us assume that $-K_2(x) \leq w^0(x) \leq K_1(x)$, $F_1^0(x) \geq 0$, $F_2^0(x) \geq 0$. Then the solution of \eqref{reduced} with initial data $(w^0,F_1^0,F_2^0)$ verifies, for all $t>0$, $x\in\RR$,
  $$
  0\leq F_1(t,x) \leq \max\left\{\|F_1^0\|_\infty, \frac{\rho\nu_1 \|K_1\|_\infty}{\delta_1}\right\}, \qquad
  0\leq F_2(t,x) \leq \max\left\{\|F_2^0\|_\infty, \frac{\rho\nu_2 \|K_2\|_\infty}{\delta_2}\right\}.
  $$
\end{corollary}

Finally, we conclude this paragraph by mentioning that the notion of sub-solutions and super-solutions for the stationary system \eqref{statF} may be deduced straightforwardly from Definition \ref{def:sub-super}~:
\begin{definition}\label{def:sub-superstat}
We say that $(\underline{F_1},\underline{F_2})$ and $(\overline{F_1},\overline{F_2})$ is a couple of sub-solution and super-solution for \eqref{statF} if
  \begin{equation}\label{def:substat}
    \left\{
      \begin{array}{l}
        - D_1 \underline{F_1}'' \leq \dfrac{\rho \nu_1}{\dfrac{b_1}{K_1} \underline{F_1} +\mu_1+\nu_1} (b_1\underline{F_1} -b_2 \overline{F_2})_+ - \delta_1 \underline{F_1}   \\[3mm]
        - D_2 \overline{F_2}'' \geq \dfrac{\rho \nu_2}{\dfrac{b_2}{K_2} \overline{F_2} +\mu_2+\nu_2} (b_2\overline{F_2} -b_1 \underline{F_1})_+ - \delta_2 \overline{F_2},
      \end{array}
    \right.
  \end{equation}
  \begin{equation}\label{def:superstat}
    \left\{
      \begin{array}{l}
        - D_1 \overline{F_1}'' \geq \dfrac{\rho \nu_1}{\dfrac{b_1}{K_1} \overline{F_1} +\mu_1+\nu_1} (b_1\overline{F_1} -b_2 \underline{F_2})_+ - \delta_1 \overline{F_1}   \\[3mm]
        - D_2 \underline{F_2}'' \leq \dfrac{\rho \nu_2}{\dfrac{b_2}{K_2} \underline{F_2} +\mu_2+\nu_2} (b_2\underline{F_2} -b_1 \overline{F_1})_+ - \delta_2 \underline{F_2}.
        \end{array}
    \right.
  \end{equation}
\end{definition}  

\subsection{A stationary solution in a homogeneous half space}

In this part, we consider that $K_1(x) = K_1$ and $K_2(x) = K_2$ are constants.
We investigate existence of a stationary solution linking the two stable steady states on the half space $(0,+\infty)$.
The aim of this section is to prove the following result~:
\begin{proposition}\label{prop:halfspace}
  Let us assume that $\Gamma_1(F_1^*)>0$ where $\Gamma_1$ is defined in Lemma \ref{lem:estim2} below.
  Then, there exists a solution of the stationary system \eqref{statF} on $(0,+\infty)$ with the boundary conditions $(F_1(0),F_2(0)) = (0,F_2^*)$ and $(F_1(+\infty),F_2(+\infty)) = (F_1^*,0)$, which is such that $F_1$ is non-decreasing and $F_2$ is non-increasing.

  We denote $(\mathbb{F}_1,\mathbb{F}_2)$ such a stationary solution.
\end{proposition}

The proof of this result is based on the construction of a sub-solution and of a super-solution. Then, we first construct sub-solution and super-solution for this system in the sense given in Definition \ref{def:sub-superstat}.
Before proving this Proposition, we set out a few technical lemmas, the proofs of which are given in the Appendix \ref{app:lem}.
\begin{lemma}\label{estim_simpl}
  Let us denote, for all $x>0$,
  $$
  \overline{F_1}(x) := F_1^* \left(1-e^{-\sqrt{\frac{\delta_1}{D_1}}x}\right), \qquad \qquad 
  \underline{F_2}(x) := F_2^* e^{-\sqrt{\frac{\delta_2}{D_2}} x}.
  $$
  Then, we have
  \begin{align*}
    - D_1 \overline{F_1}'' & \geq \frac{\rho \nu_1}{\frac{b_1}{K_1} \overline{F_1} +\mu_1+\nu_1} (b_1\overline{F_1} -b_2 \underline{F_2})_+ - \delta_1 \overline{F_1}     \\
    - D_2 \underline{F_2}'' & \leq \frac{\rho \nu_2}{\frac{b_2}{K_2} \underline{F_2} +\mu_2+\nu_2} (b_2\underline{F_2} -b_1 \overline{F_1})_+ - \delta_2 \underline{F_2}.
  \end{align*}
\end{lemma}

\begin{lemma}\label{lem:estim2}
  With the notations of Lemma \ref{estim_simpl}, let $\tilde{L}$ be the unique positive constant such that $b_1\overline{F_1}(\tilde{L}) = b_2\underline{F_2}(\tilde{L})$. We define
  \begin{equation}\label{def:F2bar}
  \overline{F_2}(x) = \left\{
    \begin{array}{ll}
      F_2^* \left(1-e^{-\sqrt{\frac{\delta_2}{D_2}}\tilde{L}} \sinh \left(\sqrt{\frac{\delta_2}{D_2}} x\right)\right), \qquad  & \text{ for } x\in (0,\tilde{L}),  \\
      F_2^* \cosh \left(\sqrt{\frac{\delta_2}{D_2}} \tilde{L}\right) e^{-\sqrt{\frac{\delta_2}{D_2}}x}, \qquad  & \text{ for } x \geq \tilde{L}.
    \end{array}
  \right.
  \end{equation}
  Let us denote
  \begin{equation}\label{eq:Gamma}
    \Gamma_1(\chi) := \int_0^{\chi} \left(\frac{\rho \nu_1}{\frac{b_1}{K_1} y + \mu_1 + \nu_1} \left(b_1 y - b_2 \overline{F_2} \left(-\sqrt{\frac{D_1}{\delta_1}}\ln\left(1-\frac{y}{F_1^*}\right)\right)\right)_+ - \delta_1 y \right)\,dy.
  \end{equation}
  Let us assume that
  \begin{equation}\label{hypGamma}
    \Gamma_1(F_1^*)>0.
  \end{equation}
  Then, there exists a non-decreasing solution, denoted $\underline{F_1}$, of the system
  \begin{equation}\label{eq:subsys}
  \left\{
    \begin{array}{l}
      - D_1 U'' = \Gamma_1'(U),   \\[4mm]
      U(0) = 0, \qquad U(+\infty) = \arg\max_{[0,F_1^*]} \Gamma_1 > 0.
    \end{array}
  \right.
\end{equation}
Moreover, $\arg\max_{[0,F_1^*]} \Gamma_1 = F_1^*$ if $\dfrac{\delta_2}{\delta_1}\geq \dfrac{D_2}{D_1}$, whereas $\arg\max_{[0,F_1^*]} \Gamma_1 < F_1^*$ if $\dfrac{\delta_2}{\delta_1} < \dfrac{D_2}{D_1}$.
\end{lemma}

\begin{remark}\label{rem:Gamma}
  It is worth noticing that thanks to the expression of $\overline{F_2}$, we have also that the expression of $\Gamma_1$ may be simplified into
  $$
  \Gamma_1(\chi) = \int_0^\chi  \left( \frac{\rho \nu_1}{\frac{b_1}{K_1} y + \mu_1 + \nu_1} \left(b_1 y - b_2 F_2^* \cosh\left(\sqrt{\frac{\delta_2}{D_2}}\tilde{L}\right) \left(1-\frac{y}{F_1^*}\right)^{\zeta}\right)_+ - \delta_1 y\right)\,dy,
  $$
  with $\zeta = \sqrt{\dfrac{\delta_2 D_1}{\delta_1 D_2}}$.  
\end{remark}

\begin{lemma}\label{lem:supsol}
  Let $\underline{F_1}$ and $\overline{F_2}$ be defined as in Lemma \ref{lem:estim2}.
  Then, we have, on $(0,+\infty)$,
  \begin{align*}
    & - D_1 \underline{F_1}'' \leq \frac{\rho \nu_1}{\frac{b_1}{K_1} \underline{F_1} +\mu_1+\nu_1} (b_1\underline{F_1} -b_2 \overline{F_2})_+ - \delta_1 \underline{F_1}     \\
    & - D_2 \overline{F_2}''  \geq \frac{\rho \nu_2}{\frac{b_2}{K_2} \overline{F_2} +\mu_2+\nu_2} (b_2\overline{F_2} -b_1 \underline{F_1})_+ - \delta_2 \overline{F_2}  \\
    & (\underline{F_1},\overline{F_2})(0) = (0,F_2^*), \qquad (\underline{F_1},\overline{F_2})(+\infty) = (\arg\max_{[0,F_1^*]} \Gamma_1 , 0).
  \end{align*}
  Moreover, we have $\underline{F_1} \leq \overline{F_1}$ and $\underline{F_2} \leq \overline{F_2}$ on $(0,+\infty)$.
\end{lemma}

We are now in position to prove the main result of this section~:  \\[2mm]
\textbf{Proof of Proposition \ref{prop:halfspace}.}
We have constructed in Lemma \ref{estim_simpl} and Lemma \ref{lem:supsol} super- and sub-solutions for the half-space problem \eqref{statF:1}--\eqref{statF:2}. The system being monotonous, it is classical to deduce by the so-called super- and sub-solutions technique that there exists a solution $(\mathbb{F}_1,\mathbb{F}_2)$ of \eqref{statF} which are such that $\underline{F_1}\leq \mathbb{F}_1 \leq \overline{F_1}$ and $\underline{F_2}\leq \mathbb{F}_2 \leq \overline{F_2}$. Notice that these inequalities imply that $\mathbb{F}_1(0) = 0$, $\mathbb{F}_2(0) = F_2^*$, and $\mathbb{F}_2(+\infty) = 0$.
We are left to verify that this solution is such that $\mathbb{F}_1$ is non-decreasing, $\mathbb{F}_2$ is non-increasing, and $\mathbb{F}_1(+\infty) = F_1^*$.

To prove the monotony, we assume by contradiction that it is not the case and that there exists $x_0>0$ and $\delta_0>0$ such that $\mathbb{F}_1'(x_0) = 0$ and $\mathbb{F}_1'(x)<0$ for $x\in (x_0,x_0+\delta_0)$ (the same argument works also for $\mathbb{F}_2$ instead of $\mathbb{F}_1$).
Then, either there exists $x_1 \geq x_0$ such that $\mathbb{F}_2'(x_1) = 0$ and $\mathbb{F}_2'(x)>0$ on $(x_1,x_1+\delta_1)$, or $\mathbb{F}_2$ is non-increasing on $(0,+\infty)$.
Then, we define,
$$
\widetilde{\mathbb{F}_1}(x) = \left\{
  \begin{array}{ll}
    \mathbb{F}_1(x),   & \text{ for } x<x_0,  \\
    \max\{\mathbb{F}_1(x_0), \mathbb{F}_1(x)\}, & \text{ for } x\geq x_0,
  \end{array}
\right.
$$
and, if $x_1<+\infty$,
$$
\widetilde{\mathbb{F}_2}(x) = \left\{
  \begin{array}{ll}
    \mathbb{F}_2(x),   & \text{ for } x<x_1,  \\
    \min\{\mathbb{F}_2(x_1), \mathbb{F}_2(x)\}, & \text{ for } x\geq x_1.
  \end{array}
\right.
$$
else, $\widetilde{\mathbb{F}_2} = \mathbb{F}_2$.
Then, $\mathbb{F}_1 \leq \widetilde{\mathbb{F}_1}$ and $\mathbb{F}_2 \geq \widetilde{\mathbb{F}_2}$.

We compute the equation verified by $\widetilde{\mathbb{F}}_1$ for $x>x_0$. We distinguish two cases. If $\widetilde{\mathbb{F}_1}(x)=\mathbb{F}_1(x)$, we have
\begin{align*}
  -D_1 \widetilde{\mathbb{F}_1}''
  & = - \mathbb{F}_1 '' = \frac{\rho \nu_1}{\frac{b_1}{K_1} \mathbb{F}_1+\mu_1+\nu_1} ( b_1 \mathbb{F}_1 - b_2 \mathbb{F}_2)_+ - \delta_1 \mathbb{F}_1  \\
  & \leq \frac{\rho \nu_1}{\frac{b_1}{K_1} \widetilde{\mathbb{F}_1}+\mu_1+\nu_1} ( b_1 \widetilde{\mathbb{F}_1} - b_2 \widetilde{\mathbb{F}_2})_+ - \delta_1 \widetilde{\mathbb{F}_1}, 
\end{align*}
since $\widetilde{\mathbb{F}_1}(x) = \mathbb{F}_1(x)$ and $\widetilde{\mathbb{F}_2}\leq \mathbb{F}_2$.
Now, if $\widetilde{\mathbb{F}_1}(x) = \mathbb{F}_1(x_0)$, then, since $x_0$ is a local maximum for $\mathbb{F}_1$, we have 
\begin{align*}
  -D_1 \widetilde{\mathbb{F}_1}''(x) & = 0 \leq -D_1 \mathbb{F}_1''(x_0) =
  \frac{\rho \nu_1}{\frac{b_1}{K_1} \mathbb{F}_1(x_0)+\mu_1+\nu_1} ( b_1 \mathbb{F}_1(x_0) - b_2 \mathbb{F}_2(x_0))_+ - \delta_1 \mathbb{F}_1(x_0)  \\
  & \leq \frac{\rho \nu_1}{\frac{b_1}{K_1} \widetilde{\mathbb{F}_1}(x)+\mu_1+\nu_1} ( b_1 \widetilde{\mathbb{F}_1}(x) - b_2 \widetilde{\mathbb{F}_2}(x))_+ - \delta_1 \widetilde{\mathbb{F}_1}(x),
\end{align*}
since $\widetilde{\mathbb{F}_1}(x) = \mathbb{F}_1(x_0)$ and $\widetilde{\mathbb{F}_2}(x) \leq  \mathbb{F}_2(x)\leq  \mathbb{F}_2(x_0)$.

By the same token, we compute the equation satisfied by $\widetilde{\mathbb{F}_2}$. If $\widetilde{\mathbb{F}_2}(x) = \mathbb{F}_2(x)$, we have
\begin{align*}
  -D_2 \widetilde{\mathbb{F}_2}''
  & = - \mathbb{F}_2 '' = \frac{\rho \nu_2}{\frac{b_2}{K_2} \mathbb{F}_2+\mu_2+\nu_2} ( b_2 \mathbb{F}_2 - b_1 \mathbb{F}_1)_+ - \delta_2 \mathbb{F}_2  \\
  & \geq \frac{\rho \nu_2}{\frac{b_2}{K_2} \widetilde{\mathbb{F}_2}+\mu_2+\nu_2} ( b_2 \widetilde{\mathbb{F}_2} - b_1 \widetilde{\mathbb{F}_1})_+ - \delta_2 \widetilde{\mathbb{F}_2}, 
\end{align*}
since $\widetilde{\mathbb{F}_2}(x) = \mathbb{F}_2(x)$ and $\mathbb{F}_1\leq \widetilde{\mathbb{F}_1}$.
Now, if $\widetilde{\mathbb{F}_2}(x) = \mathbb{F}_2(x_1)$, for $x\geq x_1 \geq x_0$, then, since $x_1$ is a local minimum for $\mathbb{F}_2$, we have 
\begin{align*}
  -D_2 \widetilde{\mathbb{F}_2}''(x) & = 0 \geq -D_2 \mathbb{F}_2''(x_1) =
  \frac{\rho \nu_2}{\frac{b_2}{K_2} \mathbb{F}_2(x_1)+\mu_2+\nu_2} ( b_2 \mathbb{F}_2(x_1) - b_1 \mathbb{F}_1(x_1))_+ - \delta_2 \mathbb{F}_2(x_1)  \\
  & \geq \frac{\rho \nu_2}{\frac{b_2}{K_2} \widetilde{\mathbb{F}_2}(x)+\mu_2+\nu_2} ( b_2 \widetilde{\mathbb{F}_2}(x) - b_1 \widetilde{\mathbb{F}_1}(x))_+ - \delta_2 \widetilde{\mathbb{F}_2}(x),
\end{align*}
since $\widetilde{\mathbb{F}_2}(x) = \mathbb{F}_2(x_1)$ and $\mathbb{F}_1(x_1) \leq  \widetilde{\mathbb{F}_1}(x_1) \leq \widetilde{\mathbb{F}_1}(x)$.

Finally, we have obtained
\begin{align*}
  &  -D_1 \widetilde{\mathbb{F}_1}'' \leq \frac{\rho \nu_1}{\frac{b_1}{K_1} \widetilde{\mathbb{F}_1}+\mu_1+\nu_1} ( b_1 \widetilde{\mathbb{F}_1} - b_2 \widetilde{\mathbb{F}_2})_+ - \delta_1 \widetilde{\mathbb{F}_1}  \\
  &  -D_2 \widetilde{\mathbb{F}_2}''  \geq \frac{\rho \nu_2}{\frac{b_2}{K_2} \widetilde{\mathbb{F}_2}+\mu_2+\nu_2} ( b_2 \widetilde{\mathbb{F}_2} - b_1 \widetilde{\mathbb{F}_1})_+ - \delta_2 \widetilde{\mathbb{F}_2}, 
\end{align*}
By comparison principle, it implies $\widetilde{\mathbb{F}_1} \leq \mathbb{F}_1$ and $\widetilde{\mathbb{F}_2} \geq \mathbb{F}_2$. It is a contradiction with the existence of $x_0$.

Thus, $\mathbb{F}_1$ is non-decreasing and $\mathbb{F}_2$ is non-increasing. Moreover they are bounded (by the super and sub-solutions) thus $\mathbb{F}_1$ admits a limit at infinity which belongs to the interval $[\arg\max_{[0,F_1^*]} \Gamma_1, F_1^*]$. Let us denote $\ell_1$ this limit, then passing into the limit into \eqref{statF:1}, this limit should be such that
$$
0 = \frac{\rho \nu_1 b_1 \ell_1 }{\frac{b_1}{K_1} \ell_1 + \mu_1 +\nu_1} - \delta_1 \ell_1,
$$
which implies $\ell_1 = 0$ or $\ell_1 = F_1^*$. Since $0$ does not belong to $[\arg\max_{[0,F_1^*]} \Gamma_1, F_1^*]$, we have $\ell_1 = F_1^*$.
It concludes the proof of Proposition \eqref{prop:halfspace}. \qed

\subsection{Invasion into a spatially homogeneous domain}\label{sec:invhom}

We consider that the study domain is spatially homogeneous, i.e. $K_1$ and $K_2$ are constants.
With the stationary solution constructed in Proposition \ref{prop:halfspace}, we define for $x\in\RR$,
\begin{equation}\label{def:calF}
  \mathcal{F}_1^0(x) =
  \left\{
    \begin{array}{ll}
      \mathbb{F}_1(x),   & \text{ for } x\geq 0,  \\
      0, & \text{ for } x<0,
    \end{array}
  \right.
  \qquad
  \mathcal{F}_2^0(x) =
  \left\{
    \begin{array}{ll}
      \mathbb{F}_2(x),   & \text{ for } x\geq 0,  \\
      F_2^*, & \text{ for } x<0,
    \end{array}
  \right.  
\end{equation}
and
$$
\mathcal{W}^0(x) = \frac{(b_1 \mathcal{F}_1^0(x)-b_2\mathcal{F}_2^0(x))_+}{\frac{b_1}{K_1}\mathcal{F}_1^0(x)+\mu_1+\nu_1} - \frac{(b_2 \mathcal{F}_2^0(x)-b_1\mathcal{F}_1^0(x))_+}{\frac{b_2}{K_2}\mathcal{F}_2^0(x)+\mu_2+\nu_2}.
$$

\begin{theorem}\label{th:invasion1}
  Let us assume that $\Gamma_1(F_1^*)>0$ where $\Gamma_1$ is defined in Lemma \ref{lem:estim2} (or equivalently in Remark \ref{rem:Gamma}).
  Let us consider system \eqref{reduced} with $K_1$ and $K_2$ constants and with initial data $(w^0,F_1^0,F_2^0)$, where $-K_2\leq w^0(x)\leq K_1$ for all $x\in\RR$.

  If there exists $x_0\in \RR$ such that for all $x\in\RR$,
  $$
  w^0(x) \geq \calW^0(x-x_0), \quad F_1^0(x)\geq \mathcal{F}_1^0(x-x_0), \quad 0\leq F_2^0(x) \leq \mathcal{F}_2^0(x-x_0).
  $$
  Then, for all $x\in\RR$, we have $\ds \lim_{t\to + \infty} (w,F_1,F_2)(t,x) = (w_1^*,F_1^*,0)$, where $w_1^* = \frac{b_1 F_1^*}{\frac{b_1}{K_1} F_1^* + \mu_1 +\nu_1}$.
\end{theorem}

\textbf{Interpretation.} This theorem states that, in spatially homogeneous domain, under some kind of \textit{``non-smallness''} assumption on the initial data $F_1^0$ and assuming that the criterion $\Gamma_1(F_1^*)>0$ holds, the species $1$ may invade the whole domain whereas species $2$ goes extinct.
We notice that this criterion is a sufficient condition and is not optimal. We left open the question of finding an optimal condition on the coefficients to guarantee invasion; finding an optimal criterion for invasion in competitive system is usually a difficult question (we refer e.g. to the review paper \cite{GirardinReview}).
The \textit{``non-smallness''} assumption in this Theorem consists in assuming that $F_1^0$ is large enough in a wide region of the domain (which is here a half space), where $F_2^0$ is very small.
Obviously, by exchanging the indices $1$ and $2$ we also have similarly that under a condition $\Gamma_2(F_2^*) >0$ and a non-smallness assumption on the initial data $F_2^0$, the species $2$ invades the whole domain whereas species $1$ goes extinct.

\medskip

\begin{proof}
  Up to a translation, we may assume that $x_0=0$.
  Let us denote $(\mathcal{W},\mathcal{F}_1,\mathcal{F}_2)$ the solution of \eqref{reduced} with initial data $(\calW^0,\mathcal{F}_1^0,\mathcal{F}_2^0)$. Then, by comparison principle (see Proposition \ref{prop:compar}), we deduce that for all $t>0$, $x\in\RR$, $w(t,x)\geq \mathcal{W}(t,x)$, $F_1(t,x)\geq \mathcal{F}_1(t,x)$, and $F_2(t,x)\leq \mathcal{F}_2(t,x)$. Moreover, by the strong maximum principle, we have for all $t>0$, $\mathcal{F}_1(t)>0$, indeed from \eqref{reduced} we have
  $$
  \partial_t \mathcal{F}_1 - D_1 \partial_{xx} \mathcal{F}_1 + \delta_1 \mathcal{F}_1 \geq 0.
  $$

  Moreover, we have by construction that $(\calW^0,\mathcal{F}_1^0,\mathcal{F}_2^0)$ verifies \eqref{def:sub}. Therefore, for all $h>0$, we deduce by comparison principle (see Proposition \ref{prop:compar}) that $\mathcal{W}(h)\geq \mathcal{W}^0$, $\mathcal{F}_1(h)\geq \mathcal{F}_1^0$, and $\mathcal{F}_2(h)\leq \mathcal{F}_2^0$.
  Applying again the comparison principle, we get that for all $t>0$ and $h>0$,  $\mathcal{W}(t+h)\geq \mathcal{W}(t)$, $\mathcal{F}_1(t+h)\geq \mathcal{F}_1(t)$, and $\mathcal{F}_2(t+h)\leq \mathcal{F}_2(t)$. Hence $t\mapsto \mathcal{W}(t)$ and $t\mapsto \mathcal{F}_1$ are non-decreasing, $t\mapsto \mathcal{F}_2(t)$ is non-increasing; and since these functions are bounded, they converge as $t\to +\infty$ towards some limit denoted respectively $\mathcal{W}^\infty$, $\mathcal{F}_1^\infty$, and $\mathcal{F}_2^\infty$. This limit verifies obviously $\mathcal{W}^\infty\geq \mathcal{W}^0$, $\mathcal{F}_1^\infty\geq \mathcal{F}_1^0$, $\mathcal{F}_1^\infty>0$, $\mathcal{F}_2^\infty\leq \mathcal{F}_2^0$ and it is a stationary solution of system \eqref{reduced}.

  Let us define
  $$
  \widetilde{y} = \sup\{y\geq 0 \text{ such that } \mathcal{F}_1^0(x+y) \leq \mathcal{F}_1^\infty(x), \ \mathcal{F}_2^0(x+y)\geq \mathcal{F}_2^\infty(x) \text{ for all } x\in\RR\}.
  $$
  Then, if $\tilde{y}<\infty$, we have on $(-\tilde{y},+\infty)$, 
  \begin{align*}
    - D_1 \partial_{xx} \mathcal{F}_1^\infty
    & = \frac{\rho\nu_1}{\frac{b_1}{K_1}\mathcal{F}_1^\infty + \mu_1 + \nu_1} \Big(b_1 \mathcal{F}_1^\infty - b_2 \mathcal{F}_2^\infty\Big)_+  - \delta_1 \mathcal{F}_1^\infty  \\
    & \geq \frac{\rho\nu_1}{\frac{b_1}{K_1}\mathcal{F}_1^\infty + \mu_1 + \nu_1} \Big(b_1 \mathcal{F}_1^\infty - b_2 \mathcal{F}_2^0(\cdot+\widetilde{y})\Big)_+  - \delta_1 \mathcal{F}_1^\infty.
  \end{align*}
  As a consequence, $\mathcal{F}_1^\infty$ is a super-solution of the equation satisfied by $\mathbb{F}_1(\cdot+\widetilde{y})$ on $(-\widetilde{y},+\infty)$. Therefore, by the strong maximum principle if there is a contact point, then $\mathcal{F}_1^\infty$ and $\mathbb{F}_1(\cdot+\widetilde{y})$ should be equal everywhere. But it is not possible since $\mathcal{F}_1^\infty(-\widetilde{y})>0 = \mathbb{F}_1(0)$.
  Hence, we have shown that $\mathbb{F}_1 (\cdot+y) \leq \mathcal{F}_1^\infty$ for all $y\in\RR$. By the same token, we also have $\mathbb{F}_2 (\cdot+y) \geq \mathcal{F}_2^\infty$ for all $y\in\RR$. It implies $\|\mathbb{F}_1\|_\infty = F_1^* \leq \mathcal{F}_1^\infty$ and $\mathcal{F}_2^\infty = 0$.
  We conclude by noting that the only stationary equation of the equation verified by $\mathcal{F}_1^\infty$ which is above $F_1^*$ is the constant $\mathcal{F}_1^\infty=F_1^*$.
\end{proof}

\subsection{Numerical illustrations}\label{sec:homogene_num}

In order to illustrate the theoretical result of Theorem \ref{th:invasion1}, we provide a numerical simulation of model \eqref{syst} when $c$ is large.
We first discretize a domain of width $100$ km with a uniform Cartesian mesh of $2000$ nodes. Then, we use a semi-implicit numerical scheme to discretize system \eqref{syst} on this mesh, i.e. the diffusive part is discretized implicitly whereas the reaction term is discretized explicitly.
We consider the numerical values for the mosquitoes populations taken from \cite[Table 3]{SBD} and \cite[Table 1]{JDE} which are given in Table \ref{table1} below.
\begin{table}[ht!]
  \centering
  \begin{tabular}{ ||c|c|c||c|c|c|| } 
    \hline
    Parameter & Value & Unit & Parameter & Value & Unit  \\
    \hline\hline
    $b_1$ & $10$ & Day$^{-1}$ & $b_2$ & $8$ & Day$^{-1}$ \\
    \hline
    $\mu_1$ & $0.03$ & Day$^{-1}$ & $\mu_2$ & $0.04$ & Day$^{-1}$ \\
    \hline
    $\nu_1$ & $0.05$ & Day$^{-1}$ & $\nu_2$ & $0.05$ & Day$^{-1}$ \\
    \hline
    $\delta_1$ & $0.04$ & Day$^{-1}$ & $\delta_2$ & $0.07$ & Day$^{-1}$ \\
    \hline
    $D_1$ & $0.025$ & km$^{2}$ Day$^{-1}$ & $D_2$ & $0.025$ & km$^{2}$ Day$^{-1}$ \\
    \hline
    $\rho$ & $0.49$ & - & $c$ & $40$ & km$^{2}$ Day$^{-1}$ \\
 \hline
\end{tabular}    
\caption{Numerical values of the parameters used in the computations.}
\label{table1}
\end{table}

We display in Figure \ref{fig:homogene} the numerical results obtained with the numerical values in Table \ref{table1}. Moreover, we choose $K_1 = 2000$ km$^{-2}$ and $K_2 = 500$ km$^{-2}$. 
With these numerical values, we find $F_1^* = 1209$ km$^{-2}$ and $F_2^* \simeq 169.4$ km$^{-2}$ and $\Gamma_1(F_1^*) \simeq 14535 > 0$. The initial data are defined by $E_1^0(x) = F_1^0(x) = 1000 \,\mathbf{1}_{\{x<-15\}}$ and $E_2^0(x) = F_2^0(x) = 150 \,\mathbf{1}_{\{x>15\}}$. 
We observe that species $1$ invades the whole domain and pushes out species $2$ as announced in Theorem \ref{th:invasion1}. Although we notice that initially, in the absence of species $1$, species $2$ seems to be invasive, as soon as the two species meet, the strong competition affects the dynamics and make species $2$ to go extinct in favor of species $1$.
\begin{figure}[h!]
\centering\includegraphics[width=0.49\linewidth,height=5.7cm]{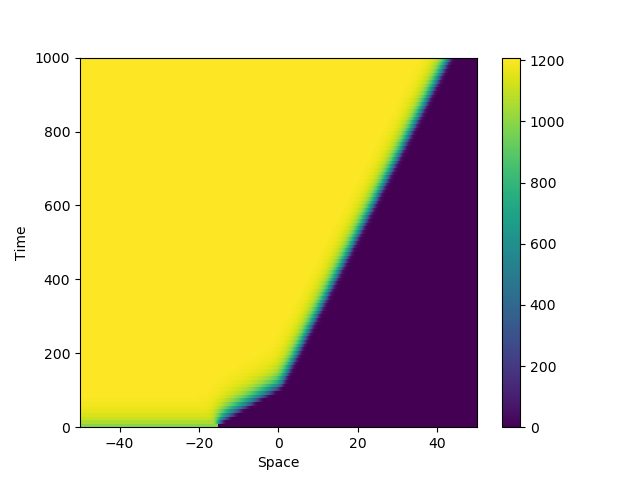} 
\includegraphics[width=0.49\linewidth,height=5.7cm]{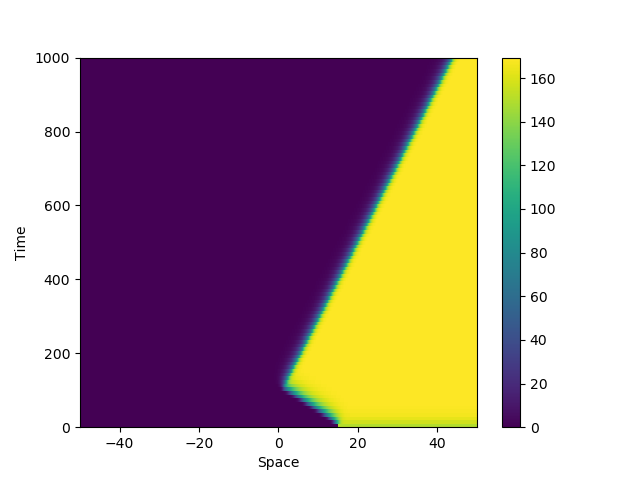}
\caption{Time dynamics of functions $F_1$ (left) and $F_2$ (right) when the condition $\Gamma_1(F_1^*)>0$ is verified. Yellow color corresponds to high density whereas blue color to low density. We observe that in this case, species $1$ invades the whole domain and pushes out species $2$. It illustrates the invasion result stated in Theorem \ref{th:invasion1}.}
\label{fig:homogene}
\end{figure}

\section{Study in a spatially heterogeneous environment}\label{sec:heterogene}

In this part, we investigate the dynamics when the carrying capacities are assumed to be spatially heterogeneous. More precisely, in the same spirit as in \cite{samia}, we consider that the carrying capacities are such that
$$
K_1(x) = K_1^F \mathbf{1}_{\{x<0\}} + K_1^U \mathbf{1}_{\{x\geq 0\}}, \quad
K_2(x) = K_2^F \mathbf{1}_{\{x<0\}} + K_2^U \mathbf{1}_{\{x\geq 0\}},
$$
modeling the fact that the left-hand side of the real line corresponds to one type of environment, for example a forest area, while the right hand side models another type of environment, for example an urban area.

We investigate system \eqref{reduced} in this heterogeneous environment.
Following the ideas in the previous Section, we first define stationary solutions and then we study the long time dynamics and the invasion of species 1 and 2 into this spatially heterogeneous environment.

\subsection{Stationary solutions in a half space}

Let us define, for $i=1,2$,
$$
F_i^{F*} = \frac{K_i^F \rho\nu_i}{\delta_i}\left(1-\dfrac{1}{\mathcal{N}_i}\right), \qquad
F_i^{U*} = \frac{K_i^U \rho\nu_i}{\delta_i}\left(1-\dfrac{1}{\mathcal{N}_i}\right),
$$
where $\mathcal{N}_i$, $i=1,2$ are defined in \eqref{def:calN}.
Then, let $\widetilde{L^F}$, respectively $\widetilde{L^U}$, be the unique positive solution to the equation
$$
b_1 F_1^{F*} \left(1-e^{-\sqrt{\frac{\delta_1}{D_1}}\widetilde{L^F}}\right) =
b_2 F_2^{F*} e^{-\sqrt{\frac{\delta_2}{D_2}} \widetilde{L^F}},
$$
respectively
$$
b_2 F_2^{U*} \left(1-e^{-\sqrt{\frac{\delta_2}{D_2}}\widetilde{L^U}}\right) =
b_1 F_1^{U*} e^{-\sqrt{\frac{\delta_1}{D_1}} \widetilde{L^U}}.
$$
Then, we introduce
\begin{equation}\label{def:Gamma1}
\Gamma_1^F(\chi) :=   \int_0^\chi  \left( \frac{\rho \nu_1}{\frac{b_1}{K_1^F} y + \mu_1 + \nu_1} \left(b_1 y - b_2 F_2^{F*} \cosh\left(\sqrt{\frac{\delta_2}{D_2}}\widetilde{L^F}\right) \left(1-\frac{y}{F_1^{F*}}\right)^{\zeta}\right)_+ - \delta_1 y\right)\,dy,
\end{equation}
with $\zeta = \sqrt{\dfrac{\delta_2 D_1}{\delta_1 D_2}}$, and
\begin{equation}\label{def:Gamma2}
\Gamma_2^U(\chi) :=   \int_0^\chi  \left( \frac{\rho \nu_2}{\frac{b_2}{K_2^U} y + \mu_2 + \nu_2} \left(b_2 y - b_1 F_1^{U*} \cosh\left(\sqrt{\frac{\delta_1}{D_1}}\widetilde{L^U}\right) \left(1-\frac{y}{F_2^{U*}}\right)^{1/\zeta}\right)_+ - \delta_2 y\right)\,dy.
\end{equation}
Then, as a consequence of Proposition \ref{prop:halfspace}, we have the following~:
\begin{lemma}\label{lem:mathbbF}
  With the above notations, we have~:
  \begin{itemize}
  \item[(i)] If $\Gamma_1^F(F_1^{F*})>0$, then there exists a solution of the stationary system \eqref{statF} on $(-\infty,0)$ complemented with boundary conditions $(F_1(-\infty),F_2(-\infty)) = (F_1^{F*},0)$, and $(F_1(0),F_2(0)) = (0,F_2^{F*})$. We denote $(\mathbb{F}_1^F,\mathbb{F}_2^F)$ such a solution, it verifies moreover that $\mathbb{F}_1^F$ is non-increasing and $\mathbb{F}_2^F$ is non-decreasing.
  \item[(ii)] If $\Gamma_2^U(F_2^{U*})>0$, then there exists a solution of the stationary system \eqref{statF} on $(0,+\infty)$ complemented with boundary conditions $(F_1(0),F_2(0)) = (F_1^{U*},0)$, and $(F_1(+\infty),F_2(+\infty)) = (0,F_2^{U*})$. We denote $(\mathbb{F}_1^U,\mathbb{F}_2^U)$ such a solution, it verifies moreover that $\mathbb{F}_1^U$ is non-increasing and $\mathbb{F}_2^U$ is non-decreasing.
  \end{itemize}
\end{lemma}

\subsection{Biological invasion result}

We are now in position to state the main result of this section~:
\begin{theorem}\label{TH2}
  Let us assume that $\Gamma_1^F(F_1^{F*})>0$ and $\Gamma_2^U(F_2^{U*})>0$ where $\Gamma_1^F$, resp. $\Gamma_2^U$, are defined in \eqref{def:Gamma1}, resp. \eqref{def:Gamma2}. Then, we have~:
  \begin{enumerate}
  \item There exist two non-increasing functions defined on $\RR$, denoted $\mathcal{F}_1^F$ and $\mathcal{F}_1^U$ with $\mathcal{F}_1^F \leq \mathcal{F}_1^U$, and two non-decreasing functions on $\RR$, denoted $\mathcal{F}_2^F$ and $\mathcal{F}_2^U$ with $\mathcal{F}_2^U \leq \mathcal{F}_2^F$, which are such that~:
    \begin{itemize}
    \item[(i)] $\mathcal{F}_1^F(x) = 0$ for $x>0$, $\lim_{x\to +\infty} \mathcal{F}_1^U(x) = 0$, and $\lim_{x\to -\infty} \mathcal{F}_1^F(x) = F_1^{F*}$, $\lim_{x\to -\infty} \mathcal{F}_1^U(x) = \max\{F_1^{U*},F_1^{F*}\}$;
    \item[(ii)] $\mathcal{F}_2^U(x) = 0$ for $x<0$, $\lim_{x\to -\infty} \mathcal{F}_2^F(x) = 0$, and $\lim_{x\to +\infty} \mathcal{F}_2^U(x) = F_2^{U*}$, $\lim_{x\to +\infty} \mathcal{F}_2^F(x) = \max\{F_2^{U*},F_2^{F*}\}$;
    \item[(iii)] $(\mathcal{F}_1^F,\mathcal{F}_2^U)$ is a sub-solution and $(\mathcal{F}_1^U,\mathcal{F}_2^F)$ is a super-solution of the stationary system \eqref{statF}.
    \end{itemize}
  \item There exist stationary solutions $(\mathbf{F}_1^{F,\infty},\mathbf{F}_2^{F,\infty})$ and $(\mathbf{F}_1^{U,\infty},\mathbf{F}_2^{U,\infty})$ such that $\mathbf{F}_1^{F,\infty} \leq \mathbf{F}_1^{U,\infty}$, $\mathbf{F}_2^{U,\infty} \leq \mathbf{F}_2^{F,\infty}$, and
  \begin{align*}
    & \lim_{x\to +\infty} (\mathbf{F}_1^{F,\infty},\mathbf{F}_2^{F,\infty})(x) = (0,F_2^{U*}), \quad  \lim_{x\to -\infty} (\mathbf{F}_1^{F,\infty},\mathbf{F}_2^{F,\infty})(x) = (F_1^{F*},0), \\
    & \lim_{x\to +\infty} (\mathbf{F}_1^{U,\infty},\mathbf{F}_2^{U,\infty})(x) = (0,F_2^{U*}), \quad  \lim_{x\to -\infty} (\mathbf{F}_1^{U,\infty},\mathbf{F}_2^{U,\infty})(x) = (F_1^{F*},0).
  \end{align*}
  \item If there exist $x_0^F>0$ and $x_0^U>0$ such that
  \begin{subequations}\label{th:coninit}
    \begin{align}
      &  \mathcal{F}_1^F (x+x_0^F) \leq F_1^0(x) \leq \mathcal{F}_1^U(x-x_0^U), \qquad
         \mathcal{F}_2^U(x-x_0^U) \leq F_2^0(x) \leq \mathcal{F}_2^F(x+x_0^F),  \\
      &  \mathcal{W}(\mathcal{F}_1^F,\mathcal{F}_2^F)(x+x_0^F) \leq w^0(x) \leq \mathcal{W}(\mathcal{F}_1^U,\mathcal{F}_2^U)(x-x_0^U),
    \end{align}
  \end{subequations}
  where we use the notation in \eqref{def:calW}. Then, the solution of \eqref{reduced} with such initial data $(w^0,F_1^0,F_2^0)$ is such that, for all $x\in\RR$, 
  \begin{align*}
    \mathbf{F}_1^{F,\infty}(x) \leq \liminf_{t\to +\infty} F_1(t,x) \leq \limsup_{t\to +\infty} F_1(t,x) \leq \mathbf{F}_1^{U,\infty}(x),  \\
    \mathbf{F}_2^{U,\infty}(x) \leq \liminf_{t\to +\infty} F_2(t,x) \leq \limsup_{t\to +\infty} F_2(t,x) \leq \mathbf{F}_2^{F,\infty}(x).    
  \end{align*}
\end{enumerate}
\end{theorem}

\textbf{Interpretation.} This result shows that, assuming that the criteria $\Gamma_1^F(F_1^{F*})>0$ and $\Gamma_2^U(F_2^{U*})>0$ holds and under some non-smallness assumption on the initial data, we have invasion of species $1$ in the region $\{x<0\}$ and invasion of the species $2$ in the region $\{x>0\}$.
As already mention after the statement of Theorem \ref{th:invasion1}, these criteria are not optimal. 
To illustrate the non-smallness condition, we provide a schematic representation of this condition in Figure \ref{fig:condini} in the situation where $F_1^{U*}<F_1^{F*}$ and $F_2^{F*}<F_2^{U*}$ (which is equivalent to $K_1^{U*}<K_1^{F*}$ and $K_2^{F*}<K_2^{U*}$).
    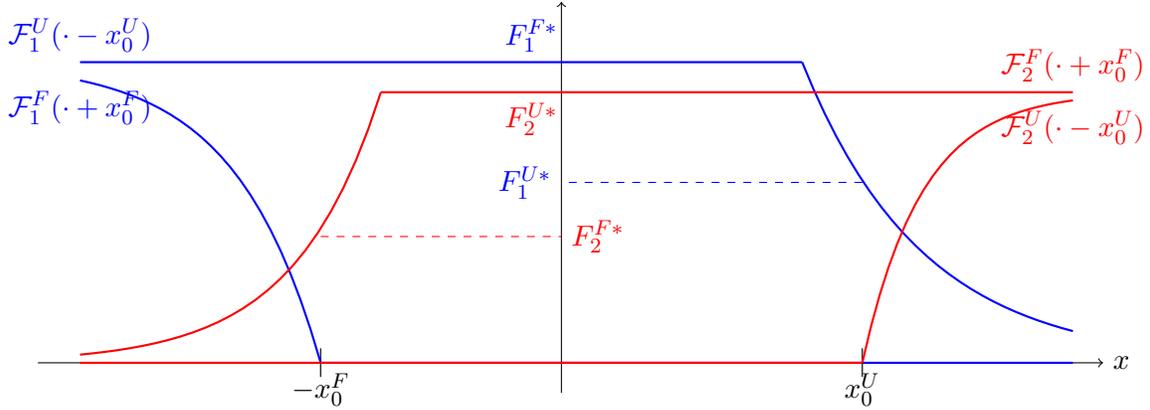
\begin{figure}[ht]
    \begin{center}
    \begin{tikzpicture}[scale=0.8]
    \draw[->] (-8.7,0) -- (9,0) node[right] {$x$};
    \draw[->] (0,-0.5) -- (0,6);
    \draw (-4,0) node {$|$} node[below] {$-x_0^F$};
    \draw[color=blue, thick ,domain=-4:-8] plot (\x,{5*(1-exp(0.7*(\x+4)))}) node[below] {$\mathcal{F}_1^F(\cdot+x_0^F)$};
    \draw[color=blue, thick, domain=-4:8.5] plot (\x,{0});
    \draw (5,0) node {$|$} node[below] {$x_0^U$};
    \draw[color=blue,thick,domain=4:-8] plot (\x,{5}) node[above] {$\mathcal{F}_1^U(\cdot-x_0^U)$};
    \draw[color=blue,thick,domain= 4:8.5] plot (\x,{5*exp(0.5*(4-\x))});
    \draw[color=blue] (-0.5,5) node[above] {$F_1^{F*}$};
    \draw[color=red,thick,domain=-3:8.5] plot (\x,{4.5}) node[above] {$\mathcal{F}_2^F(\cdot+x_0^F)$};
    \draw[color=red,thick,domain=-8:-3] plot (\x,{4.5*exp(0.7*(\x+3))});
    \draw[color=red] (-0.5,4.5) node[below] {$F_2^{U*}$};
    \draw[-,dashed,color=blue] (5,3.) -- (0,3.) node[left] {$F_1^{U*}$};
    \draw[color=red,thick,domain=-8:5] plot (\x,{0});
    \draw[color=red,thick,domain=5:8.5] plot (\x,{4.5*(1-exp(5-\x))}) node[below] {$\mathcal{F}_2^U(\cdot-x_0^U)$};
    \draw[-,dashed,color=red] (-4,2.1) -- (0,2.1) node[right] {$F_2^{F*}$};
    \end{tikzpicture} 
    \caption{Schematic representations of the conditions in Theorem \ref{TH2} on the initial data in the case where $F_1^{U*}<F_1^{F*}$ and $F_2^{F*}<F_2^{U*}$~: the initial conditions $F_1^0$ should be comprise between the two blue curves, the initial condition $F_2^0$ should be comprise between the two red curves.}\label{fig:condini}
    \end{center}
    \end{figure}

\begin{proof}
  The idea relies on a construction of sub-solution and super-solution as defined in Definition \ref{def:sub-super} and to use the comparison principle Proposition \ref{prop:compar}. To do so, we will use the stationary solutions $(\mathbb{F}_1^F,\mathbb{F}_2^F)$ and $(\mathbb{F}_1^U,\mathbb{F}_2^U)$ in Lemma \ref{lem:mathbbF}.

  Let us assume first that $F_1^{F*}>F_1^{U*}$. Then, there exists $\alpha^U>0$ such that there exists a non-increasing solution $\widetilde{\mathbb{F}_1^U}$ of
  \begin{align*}
    &- D_1 y'' = \frac{\rho\nu_1b_1 y}{\frac{b_1}{K_1^U} y + \mu_1 + \nu_1} - \delta_1 y,  \\
    & y(0) = F_1^{F*}, \quad y(\alpha^U) = F_1^{U*}, \quad y'(\alpha^U) = (\mathbb{F}_1^U)'(0).
  \end{align*}
  Indeed, let $b>0$ and $\widetilde{F_1}$ the solution of this differential equation with initial data $\widetilde{F_1}(b)=F_1^{U*}$, $\widetilde{F_1}'(b)=(\mathbb{F}_1^U)'(0)$. Then, on $(0,b)$, the function $\widetilde{F_1}$ is convex since on this interval we have $\widetilde{F_1}\geq F_1^{U*}$ and from the equation we deduce $-D_1 \widetilde{F_1} \geq 0$ when $\widetilde{F_1}\geq F_1^{U*}$.
  Then, for $b$ large enough, there exists $a\geq 0$ such that $\widetilde{F_1}(a)=F_1^{F*}$.
  Now, we take $\widetilde{F_1^U}(x) = F_1(x+a)$ and $\alpha^U=b-a$.

  Let us then define $\mathcal{F}_1^U$ and $\mathcal{F}_2^U$ in the following way~:
  \begin{itemize}
  \item If $F_1^{F*}\leq F_1^{U*}$,
    $$
    \mathcal{F}_1^U(x) =
    \left\{
      \begin{array}{ll}
        F_1^{U*} = \max\{F_1^{U*},F_1^{F*}\}, \quad & \text{ for } x<0,  \\[1mm]
        \mathbb{F}_1^U(x), \quad & \text{ for } x\geq 0,
      \end{array}
    \right.
    $$
    and
    $$
    \mathcal{F}_2^U(x) =
    \left\{
      \begin{array}{ll}
        0, \qquad & \text{ for } x<0,  \\[1mm]
        \mathbb{F}_2^U(x), \qquad & \text{ for } x\geq 0.
      \end{array}
    \right.
    $$
  \item If $F_1^{F*}> F_1^{U*}$,
    $$
    \mathcal{F}_1^U(x) =
    \left\{
      \begin{array}{ll}
        F_1^{F*} = \max\{F_1^{U*},F_1^{F*}\}, \quad & \text{ for } x<0,  \\[1mm]
        \widetilde{\mathbb{F}_1^U}(x), \quad & \text{ for } x\in [0,\alpha^U], \\[1mm]
        \mathbb{F}_1^U(x-\alpha^U), \quad & \text{ for } x> \alpha^U,
      \end{array}
    \right.
    $$
    and
    $$
    \mathcal{F}_2^U(x) =
    \left\{
      \begin{array}{ll}
        0, \qquad & \text{ for } x<\alpha^U,  \\[1mm]
        \mathbb{F}_2^U(x-\alpha^U), \qquad & \text{ for } x\geq \alpha^U.
      \end{array}
    \right.
    $$
  \end{itemize}

  We proceed in a similar way to construct $\mathcal{F}_1^F$ and $\mathcal{F}_2^F$.
  First, by the same token as above, if $F_2^{F*}<F_2^{U*}$, there exists $\alpha^F<0$ and a non-decreasing solution denoted $\widetilde{\mathbb{F}_2^F}$ of the equation
  \begin{align*}
    &- D_2 y'' = \frac{\rho\nu_2b_2 y}{\frac{b_2}{K_2^F} y + \mu_2 + \nu_2} - \delta_2 y,  \\
    & y(0) = F_2^{U*}, \quad y(\alpha^F) = F_2^{F*}, \quad y'(\alpha^F) = (\mathbb{F}_2^F)'(0).
  \end{align*}
  Then, we construct $\mathcal{F}_1^F$ and $\mathcal{F}_2^F$ in the following way~:
  \begin{itemize}
  \item If $F_2^{F*}\geq F_2^{U*}$,
    $$
    \mathcal{F}_1^F(x) =
    \left\{
      \begin{array}{ll}
        \mathbb{F}_1^F(x), \quad & \text{ for } x<0,  \\[1mm]
        0, \quad & \text{ for } x\geq 0,
      \end{array}
    \right.
    $$
    and
    $$
    \mathcal{F}_2^F(x) =
    \left\{
      \begin{array}{ll}
        \mathbb{F}_2^F(x), \qquad & \text{ for } x<0,  \\[1mm]
        F_2^{F*} = \max\{F_2^{U*},F_2^{F*}\}, \qquad & \text{ for } x\geq 0.
      \end{array}
    \right.
    $$
  \item If $F_2^{F*}< F_2^{U*}$,
    $$
    \mathcal{F}_1^F(x) =
    \left\{
      \begin{array}{ll}
        \mathbb{F}_1^F(x-\alpha^F), \qquad & \text{ for } x < \alpha^F,  \\[1mm]
        0, \qquad & \text{ for } x \geq \alpha^F,
      \end{array}
    \right.
    $$
    and
    $$
    \mathcal{F}_2^F(x) =
    \left\{
      \begin{array}{ll}
        \mathbb{F}_2^F(x-\alpha^F), \quad & \text{ for } x < \alpha^F,  \\[1mm]
        \widetilde{\mathbb{F}_2^F}(x), \quad & \text{ for } x\in [\alpha^F,0], \\[1mm]
        F_2^{U*} = \max\{F_2^{U*},F_2^{F*}\}, \quad & \text{ for } x>0.
      \end{array}
    \right.
    $$
  \end{itemize}

  By construction, $(\mathcal{F}_1^F,\mathcal{F}_2^U)$ and $(\mathcal{F}_1^U,\mathcal{F}_2^F)$ are respectively sub-solution and super-solution for the stationary system \eqref{statF}.
  Indeed, this is a consequence of the facts that~: Firstly, $(\mathbb{F}_1^F,\mathbb{F}_2^F)$ and $\mathbb{F}_1^U,\mathbb{F}_2^U)$ are stationary solution respectively on $(-\infty,0)$ and $(0,+\infty)$. Secondly, $0$ is a sub-solution for the stationary system whereas $\max\{F_1^{U*},F_1^{F*}\}$, respectively $\max\{F_2^{U*},F_2^{F*}\}$, is a super-solution for \eqref{statF:1}, respectively \eqref{statF:2}. Thirdly, $\mathcal{F}_1^U$ and $\mathcal{F}_1^U$ are non-increasing, $\mathcal{F}_2^U$ and $\mathcal{F}_2^U$ are non-decreasing.
  We deduce that for all $x_0^F>0$ and $x_0^U>0$, we have that $x\mapsto (\mathcal{W}(\mathcal{F}_1^F,\mathcal{F}_2^F)(x+x_0^F), \mathcal{F}_1^F(x+x_0^F),\mathcal{F}_2^U(x-x_0^U))$ and $x\mapsto (\mathcal{W}(\mathcal{F}_1^U,\mathcal{F}_2^U)(x-x_0^U),\mathcal{F}_1^U(x-x_0^U),\mathcal{F}_2^F(x+x_0^F))$ is respectively a sub-solution and super-solution of \eqref{reduced} in the sense of Definition \ref{def:sub-super}. It concludes the proof of the point 1.

  \smallskip

  Let us take $x_0^F>0$ and $x_0^U>0$ such that
  $$
  \mathcal{F}_1^F(\cdot+x_0^F) \leq \mathcal{F}_1^{U}(\cdot-x_0^U), \quad
  \mathcal{F}_2^U(\cdot-x_0^U) \leq \mathcal{F}_2^F(\cdot+x_0^F),
  $$
  such $x_0^F>0$ and $x_0^U>0$ always exist since $\mathcal{F}_1^F\leq \mathcal{F}_1^U$ and are non-increasing, and $\mathcal{F}_2^U\leq \mathcal{F}_2^F$ and are non-decreasing (see Figure \ref{fig:condini} for an illustration).
  We denote $(\mathbf{W}^F,\mathbf{F}_1^F,\mathbf{F}_2^U)$, respectively $(\mathbf{W}^U,\mathbf{F}_1^U,\mathbf{F}_2^F)$, the solution to \eqref{reduced} with initial data $(\mathcal{W}(\mathcal{F}_1^F,\mathcal{F}_2^F)(\cdot+x_0^F), \mathcal{F}_1^F(\cdot+x_0^F),\mathcal{F}_2^U(\cdot-x_0^U))$, respectively $(\mathcal{W}(\mathcal{F}_1^U,\mathcal{F}_2^U)(\cdot-x_0^U),\mathcal{F}_1^U(\cdot-x_0^U),\mathcal{F}_2^F(\cdot+x_0^F))$.
  By the same token as in the proof of Theorem \ref{th:invasion1}, we get that $t\mapsto (\mathbf{W}^F,\mathbf{F}_1^F,\mathbf{F}_2^U)$ is non-decreasing and $t\mapsto (\mathbf{W}^U,\mathbf{F}_1^U,\mathbf{F}_2^F)$ is non-increasing. By comparison principle (see Proposition \ref{prop:compar}), we also have for all $t>0$ and $x\in\RR$,
  $$
  \mathcal{F}_1^F(\cdot+x_0^F) \leq \mathbf{F}_1^F(t,x)\leq \mathbf{F}_1^U(t,x)\leq \mathcal{F}_1^{U}(\cdot-x_0^U)
  $$
  and
  $$
  \mathcal{F}_2^U(\cdot-x_0^U) \leq \mathbf{F}_2^U(t,x) \leq \mathbf{F}_2^F(t,x)\leq \mathcal{F}_2^F(\cdot+x_0^F).
  $$
  Then, as $t$ goes to infinity, $t\mapsto (\mathbf{W}^F,\mathbf{F}_1^F,\mathbf{F}_2^U)$ converges towards some stationary solutions denoted $(\mathbf{W}^{F,\infty},\mathbf{F}_1^{F,\infty},\mathbf{F}_2^{U,\infty})$, and similarly $t\mapsto (\mathbf{W}^U,\mathbf{F}_1^U,\mathbf{F}_2^F)$ converges towards a stationary solution $(\mathbf{W}^{U,\infty},\mathbf{F}_1^{U,\infty},\mathbf{F}_2^{F,\infty})$. Moreover, we have by passing into the limit in the above inequalities
  \begin{align}
    & \mathcal{F}_1^F(\cdot+x_0^F) \leq \mathbf{F}_1^{F,\infty}(x) \leq \mathbf{F}_1^{U,\infty}(x) \leq \mathcal{F}_1^{U}(\cdot-x_0^U), \label{estim1inf}\\
    & \mathcal{F}_2^U(\cdot-x_0^U) \leq \mathbf{F}_2^{U,\infty}(x) \leq \mathbf{F}_2^{F,\infty}(x) \leq \mathcal{F}_2^F(\cdot+x_0^F). \label{estim2inf}
  \end{align}
  Hence, $\lim_{x\to +\infty} F_1^{F,\infty} = 0$, $\lim_{x\to +\infty} F_1^{U,\infty} = 0$, $\lim_{x\to -\infty} F_2^{F,\infty} = 0$, and  $\lim_{x\to -\infty} F_2^{U,\infty} = 0$.
  Let us assume that there exists $\xi>0$ such that $\mathbf{F}_2^{F,\infty}(\xi)>F_2^{U*}$. Then, from \eqref{statF:2}
  \begin{align*}
    -D_2 \partial_{xx}\mathbf{F}_2^{F,\infty}(\xi) \leq
    & \frac{\rho\nu_2 b_2 \mathbf{F}_2^{F,\infty}(\xi)}{\frac{b_2}{K_2^U} \mathbf{F}_2^{F,\infty}(\xi) + \mu_2 + \nu_2} - \delta_2 \mathbf{F}_2^{F,\infty}(\xi)  \\
    & = \rho\nu_2 b_2 \mathbf{F}_2^{F,\infty}(\xi) \left(\frac{1}{\frac{b_2}{K_2^U} \mathbf{F}_2^{F,\infty}(\xi) + \mu_2 + \nu_2} - \frac{1}{\frac{b_2}{K_2^U} F_2^{U*} + \mu_2 + \nu_2}\right) \leq 0.
  \end{align*}
  Therefore, $\mathbf{F}_2^{F,\infty}$ is convex as soon as $\mathbf{F}_2^{F,\infty}>F_2^{U*}$, then, if $\partial_x \mathbf{F}_2^{F,\infty}(\xi)>0$, we deduce that $\mathbf{F}_2^{F,\infty}$ is increasing on $(\xi,+\infty)$, thus cannot be bounded. It is a contradiction. Therefore, either $\mathbf{F}_2^{F,\infty}\leq F_2^{U*}$ which implies with \eqref{estim2inf} that $\lim_{x\to +\infty} \mathbf{F}_2^{F,\infty}(x)=F_2^{U*}$, or $x\mapsto \mathbf{F}_2^{F,\infty}(x)$ is non-increasing on $(0,+\infty)$ hence it admits a limit as $x$ goes to $+\infty$ which is a equilibrium for \eqref{statF:2}, the only possibility is $F_2^{U*}$.
  By the same token, we show that $\lim_{x\to - \infty} \mathbf{F}_1^{U,\infty}(x) = F_1^{F*}$.
  The proof of the point 2 is done.
  
  \smallskip

  The point 3 is a consequence of the comparison principle (see Proposition \ref{prop:compar}). Indeed, we deduce from this comparison principle that if \eqref{th:coninit} holds, then for any $t>0$ and $x\in \RR$, we have
  $$
  \mathbf{F}_1^F(t,x) \leq F_1(t,x) \leq \mathbf{F}_1^U(t,x), \quad
  \mathbf{F}_2^U(t,x) \leq F_2(t,x) \leq \mathbf{F}_2^F(t,x).
  $$
  We conclude by taking the limit as $t$ goes to $+\infty$.
\end{proof}
    
\subsection{Numerical illustration}\label{sec:heterogene_num}

Finally, we illustrate our theoretical results in Theorem \ref{TH2} with a numerical simulation. We use the numerical values in Table \ref{table1} for the mosquitoes populations and the same numerical discretization of system \ref{syst} as in Section \ref{sec:homogene_num}.
We recall that $K_i(x) = K_i^F \mathbf{1}_{\{x<0\}} + K_i^U \mathbf{1}_{\{x\geq 0\}}$ and we choose
$$
K_1^F = 2000 \text{km}^{-2}, \quad K_1^U = 300 \text{km}^{-2}, \quad
K_2^F = 500 \text{km}^{-2}, \quad K_2^U = 2200 \text{km}^{-2}.
$$
With these numerical values, we compute $F_1^{F*} = 1209$ km$^{-2}$, $F_2^{U*} = 745.25$ km$^{-2}$, $\Gamma_1^F(F_1^{F*}) \simeq 14535 > 0$, and $\Gamma_2^U(F_2^{U*}) \simeq 368.3 >0 $. The initial data are defined by $E_1^0(x) = F_1^0(x) = 1000 \,\mathbf{1}_{\{x<-15\}}$ and $E_2^0(x) = F_2^0(x) = 700 \,\mathbf{1}_{\{x>15\}}$. 
We display in Figure \ref{fig:heterogene} the numerical results obtained.
We observe that species $1$ invades the domain $\{x<0\}$ and species $2$ invades the domain $\{x>0\}$. Then each species remains in its habitat and a segregation occurs. 
\begin{figure}[h!]
\centering\includegraphics[width=0.49\linewidth,height=5.7cm]{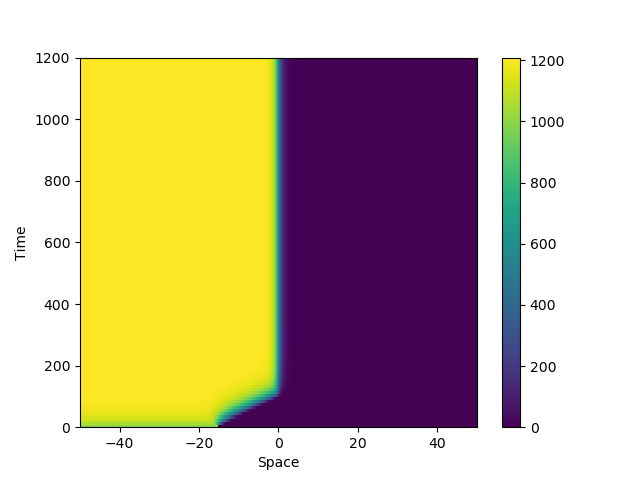} 
\includegraphics[width=0.49\linewidth,height=5.7cm]{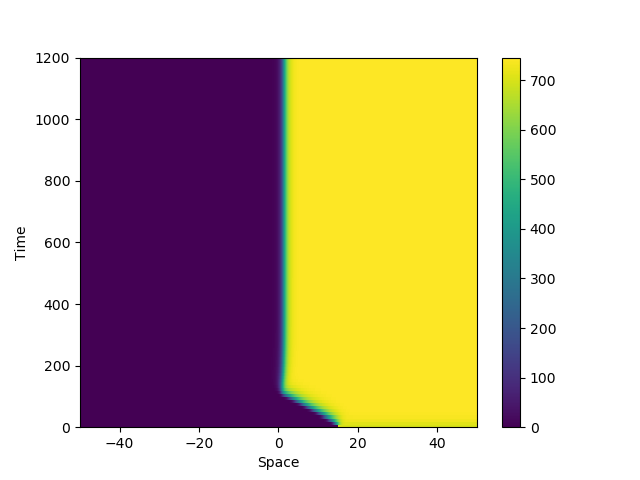}
\caption{Time dynamics of functions $F_1$ (left) and $F_2$ (right) when the condition $\Gamma_1^F(F_1^*)>0$ is verified. Yellow color corresponds to high density, blue color to low density. In this case, species $1$ invades the whole domain and pushes out species $2$. It illustrates the invasion result stated in Theorem \ref{TH2}.}
\label{fig:heterogene}
\end{figure}

\section{Conclusion}

Based on the observation that \textit{Aedes} mosquitoes populations are segregated in different habitats, we have proposed a mathematical model for two populations of mosquitoes with competition at larval stage. The model consists of a competitive system of reaction-diffusion equations. Assuming a high degree of competition, we first reduce the system. For this reduced system, we have shown that under certain conditions on the parameters and assuming sufficiently large initial data, each population settles in the habitat that is most favorable to it. This model therefore reproduces observations well. This result extends the one proposed in \cite{samia} where a simplified scalar reaction-diffusion system is investigated.

From a mathematical point of view the strong competition asymptotic allowed us to reduce the dimensionality of the problem. Then a natural perspective is to investigate the full system \eqref{syst}. General conditions to guarantee invasion of one species in a competitive system of reaction-diffusion equation are, up to our knowledge, not known. Then, the first question is to study whether the technique of construction of sub-solutions and super-solutions developed in this paper could be applied to obtain sufficient conditions on the parameters to know which population will invade the other.

Another perspective for practical applications is to understand more carefully the dynamics of the competition between mosquitoes at larval stage. In model \eqref{syst} the competition is modelled in a simple manner. But, it surely depends on the size of the larvae and on the resources available. Thus, the mathematical model should be complexified to have a better description of this competition. Obviously the mathematical analysis of the resulting system will be more challenging.

\newpage

\appendix
 
\section{Strong competition asymptotic}

In this appendix, we justify rigorously the reduction of system \eqref{syst} towards system \eqref{eq:w}--\eqref{systF} when the competition is assumed to be very large, following the ideas developed in \cite{Dancer1,Girardin,Perthame,samia}.
Let us set $c=\frac{1}{\eps}$ and we let $\eps\to 0$ in system \ref{syst}. Then it reads
\begin{subequations}\label{systeps}
\begin{align}\label{systeps:E1}
&\partial_t E_{1,\eps} = b_1 F_{1,\eps} \left(1 - \frac{E_{1,\eps}}{K_1}\right) - \frac{1}{\eps} E_{1,\eps} E_{2,\eps} - (\mu_1+\nu_1) E_{1,\eps} \\
\label{systeps:F1}
&\partial_t F_{1,\eps} - D_1 \partial_{xx} F_{1,\eps} = \rho \nu_1 E_{1,\eps} - \delta_1 F_{1,\eps}  \\
\label{systeps:E2}
&\partial_t E_{2,\eps} = b_2 F_{2,\eps} \left(1 - \frac{E_{2,\eps}}{K_2}\right) - \frac{1}{\eps} E_{1,\eps} E_{2,\eps} - (\mu_2+\nu_2) E_{2,\eps} \\
\label{systeps:F2}
&\partial_t F_{2,\eps} - D_2 \partial_{xx} F_{2,\eps} = \rho \nu_2 E_{2,\eps} - \delta_2 F_{2,\eps}.
\end{align}
\end{subequations}
This system is complemented with non-negative initial conditions $E_{1,\eps}^0, F_{1,\eps}^0, E_{2,\eps}^0, F_{2,\eps}^0$.
We assume that these initial data are non-negative, continuous on $\RR$ and uniformly bounded in $W^{1,1}(\mathbb{R})\cap L^\infty(\mathbb{R})$~:
\begin{align} \label{54}
\left\|E_{1,\varepsilon}^0\right\|_{L^1\left(\mathbb{R}\right)}+\left\|E_{1,\varepsilon}^0\right\|_{L^{\infty}\left(\mathbb{R}\right)}+\left\|E_{2,\varepsilon}^0\right\|_{L^1 \left(\mathbb{R}\right)}+\left\|E_{2,\varepsilon}^0\right\|_{L^{\infty}\left(\mathbb{R}\right)} \leq C^0,
\end{align}
\begin{align} \label{53}
\left\|\partial_x E_{1,\varepsilon}^0\right\|_{L^1\left(\mathbb{R} \right)}+\left\|\partial_x E_{2,\varepsilon}^0\right\|_{L^1\left(\mathbb{R} \right)} \leq C^1. 
\end{align}
From these assumptions, we may extract a subsequence still denoted by ($E_{1,\varepsilon}^0, E_{2,\varepsilon}^0)$ of initial data that convergences strongly in $L^1(\mathbb{R}) \times L^1(\mathbb{R})$. We assume moreover that the limiting initial data are segregated, i.e.
\begin{equation}\label{assumpini}
E_{1,\varepsilon}^0 \rightarrow E_1^0, \quad E_{2,\varepsilon}^0 \rightarrow E_2^0 \quad \text {  in } L^1\left(\mathbb{R}\right), \text{ and }
E_1^0 E_2^0 = 0 .
\end{equation}
The carrying capacities is assumed to be such that, for $i=1,2$,
\begin{equation}\label{hyp:K1K2}
0<K_i \in L^\infty(\RR), \qquad \frac{1}{K_i} \in L^\infty(\RR) \cap BV(\RR).
\end{equation}
In this Appendix, we prove the following result~:
\begin{theorem}\label{th:reduc}
Let $T>0$ and assume that \eqref{hyp:K1K2} holds. Let us consider system \eqref{systeps} with assumptions \eqref{54}-\eqref{assumpini} on the initial data. Then, as $\varepsilon\to 0$, we have for $i=1,2$,
$$
E_{i,\varepsilon} \rightarrow E_i, \quad F_{i,\varepsilon} \rightarrow F_i \quad \text { strongly in } L^1\left([0,T], L^1(\mathbb{R})\right),
$$
where $E_1, E_2,F_1,F_2 \in L^{\infty}\left(\mathbb{R}^{+} ; L^1 \cap L^{\infty}\left(\mathbb{R} \right)\right)$,  $E_1 E_2(t, x)=0$  a.e. and $w =E_1-E_2$ solves 
\begin{equation}\label{eqap:w}
\partial_t w = b_1 F_1 \left(1-\dfrac{w_+}{K_1(x)}\right)-(\mu_1+\nu_1) w_+ - b_2 F_2 \left(1-\dfrac{w_-}{K_2(x)}\right) + (\mu_2+\nu_2) \omega_-,
\end{equation}
and $F_1$ and $F_2$ satisfy
\begin{subequations}\label{app:systF}
\begin{align}\label{app:systF:1}
&\partial_t F_1 - D_1 \Delta F_1 = \rho \nu_1 w_+ - \delta_1 F_1  \\
\label{app:systF:2}
&\partial_t F_2 - D_2 \Delta F_2 = \rho \nu_2 w_- - \delta_2 F_2.
\end{align}
\end{subequations}
In this system, we use the standard notations for the positive and negative parts, i.e. $u_+ = \max\{0,u\}$, and $u_- = \max\{0,-u\}$.
This equation is complemented with initial data
$$
w(t=0) = w^0(x) := E_{1}^0(x) -  E_{2}^0(x), \qquad
F_1(t=0) = F_1^0, \qquad F_2(t=0) = F_2^0.
$$
\end{theorem}
\begin{proof}
  We do not give all details of the proof but explain how the proof in \cite{samia} can be adapted to the system studied in this paper.
The proof is split into several steps.

\textbf{ Step 1: Uniform a priori estimates  } \\
From standard results, solutions to system \eqref{systeps} are non-negative and uniformly bounded in $L^\infty(\RR^+,L^\infty(\RR))$. Then, integrating system \eqref{systeps}, we obtain, for $i=1,2$,
\begin{align*}
 \left\{\begin{array}{l}
\ds\frac{d}{dt}\int_{\mathbb{R}} E_{i,\varepsilon}(t, x) d x + \int_{\mathbb{R}} \frac{E_{1,\varepsilon}(t, x) E_{2,\varepsilon}(t, x)}{\varepsilon} \,dx  \\[3mm]
\qquad +  \ds \int_{\mathbb{R}} \big( b_i \frac{E_{i,\varepsilon}(t, x) F_{i,\eps}(t,x) }{K_i(x)} + (\mu_i+\nu_i) E_{i,\varepsilon}(t,x) \big)\,dxds 
\leq  \ds\int_{\mathbb{R}} b_i F_{i,\varepsilon}(t,x) d x, \\[4mm] 
\ds\frac{d}{dt} \int_{\mathbb{R}} F_{i,\varepsilon}(t, x) d x + \int_{\mathbb{R}} \delta_i F_{i,\eps}(t,x)\,dx  \leq \int_{\RR} \rho \nu_i E_{i,\eps}(t,x)\,dx.
\end{array}\right.
\end{align*}
Adding the two inequalities and applying a Gronwall lemma, we deduce the uniform bounds in $L^\infty([0,T],L^1(\RR))$ uniform with respect to $\eps$. Moreover, we also have the uniform bound 
\begin{equation}\label{estim1}
\int_0^T \int_\RR \frac{E_{1,\eps}(t,x) E_{2,\eps}(t,x)}{\eps} \,dxdt \leq C.
\end{equation}
By parabolic regularity, we get that $F_{i,\eps}\in L^\infty_{loc}(\RR^+,C^1(\RR))\cap L^\infty_{loc}(\RR^+,W^{1,1}(\RR))$, for $i=1,2$, hence we deduce that $E_{i,\eps} \in L^\infty_{loc}(\RR^+,C^1(\RR))\cap L^\infty_{loc}(\RR^+,W^{1,1}(\RR))$, for $i=1,2$.

Next, differentiating system \eqref{systeps} with respect to $x$ and multiplying be the sign function, we get, for $i=1,2$,
\begin{align*}
  \partial_t |\partial_x E_{i,\eps}| =
  & b_i \partial_x F_{i,\eps} \left(1-\frac{E_{i,\eps}}{K_i}\right) \operatorname{sgn}(\partial_x E_{i,\eps}) - \frac{1}{\eps} \left(\partial_x E_{1,\eps} E_{2,\eps} + E_{1,\eps} \partial_x E_{2,\eps}\right)\operatorname{sgn}(\partial_x E_{i,\eps})  \\
  & - \left(\frac{b_i}{K_i} F_i + \mu_i + \nu_i\right) |\partial_x E_i|. 
\end{align*}
Moreover, we have
\begin{align*}
& \Big(\partial_x E_{1,\varepsilon} E_{2,\varepsilon}+ E_{1,\varepsilon} \partial_x E_{2,\varepsilon}\Big)\Big(\operatorname{sgn}(\partial_x E_{1,\varepsilon})+\operatorname{sgn}(\partial_x E_{2,\varepsilon})\Big)  \\
& =  \Big(\left|\partial_x E_{1,\varepsilon}\right| + \partial_x E_{1,\varepsilon} \operatorname{sgn}(\partial_x E_{2,\varepsilon})\Big) E_{2,\varepsilon}
+\Big(|\partial_x E_{2,\varepsilon}| + \partial_x E_{2,\varepsilon} \operatorname{sgn}(\partial_x E_{1,\varepsilon})\Big) E_{1,\varepsilon} \geq  0.
\end{align*}
Injecting into the latter inequality after adding the equations for $i=1$ and $i=2$, we obtain
\begin{equation}\label{eq:dxEi}
\partial_t \Big(|\partial_x E_{1,\eps}|+|\partial_x E_{2,\eps}|\Big) \leq 
\sum_{i=1}^2 b_i |\partial_x F_{i,\eps}| \left|1-\frac{E_{i,\eps}}{K_i}\right| \leq C \Big(|\partial_x F_{1,\eps}|+|\partial_x F_{2,\eps}|\Big),
\end{equation}
where we use the uniform bound of $E_{i,\eps}$ in $L^\infty$ in the last inequality.

Le us recall the Kato's inequality~: $\Delta |f| \geq \operatorname{sgn}(f) \Delta f$ in the sense of distribution for any integrable function $f$ whose Laplacian is integrable.
Then, differentiating equations \eqref{systeps:F1} and \eqref{systeps:F2} with respect to $x$ and multiplying with the sign function, we obtain, for $i=1,2$,
$$
\partial_t F_{i,\eps} - D_i \Delta |\partial_x F_i| \leq \rho \nu_i |\partial_x E_i| - \delta_i |\partial_x F_i|.
$$
Integrating and adding equation \eqref{eq:dxEi} after integrating it, we get, for some positive constant $C$,
$$
\frac{d}{dt} \int_\RR \Big(|\partial_x E_{1,\eps}| + |\partial_x E_{2,\eps}| + |\partial_x F_{1,\eps}| + |\partial_x F_{2,\eps}|\Big)\,dx \leq
C \int_\RR \Big(|\partial_x E_{1,\eps}| + |\partial_x E_{2,\eps}| + |\partial_x F_{1,\eps}| + |\partial_x F_{2,\eps}|\Big)\,dx.
$$
We conclude by a Gronwall argument that the quantities $E_{1,\eps}$, $E_{2,\eps}$, $F_{1,\eps}$, $F_{2,\eps}$ are uniformly bounded in $L^1([0,T],W^{1,1}(\RR))$. \\

\textbf{ Step 2: Compactness and passing to the limit} \\
This step follows closely the proof in \cite{samia}. First, using the Aubin-Lions lemma and the compact injection of $W^{1,1}(\Omega)$ into $L^1(\Omega)$ for $\Omega$ an open bounded domain, we obtain the local strong compactness of the sequences $(E_{1,\eps})\eps$ and $(E_{2,\eps})_\eps$.
To pass from local to global convergence, it remains to prove that the mass in the tail is uniformly small with respect to $\eps$, which follows the same argument as in \cite{samia}.
It is classical to deduce from the linear equations \eqref{systeps:F1} and \eqref{systeps:F2} that the sequences $(F_{1,\eps})$ and $(F_{2,\eps})$ also converge strongly.
As a consequence, up to an extraction, we have strong convergence towards limit denoted respectively $E_1$, $E_2$, $F_1$, $F_2$.

To pass to the limit, we first observe that from the estimate \eqref{estim1} that when $\eps\to 0$, the limit verifies the segregation properties $E_1 E_2 = 0$ for a.e. $t>0$ and $x$ and also for $t=0$ thanks to assumption \eqref{assumpini}.
Next, we define $w_\eps = E_{1,\eps}-E_{2_\eps} \to w := E_1 - E_2$. From the segregation property, we deduce $w_+ = E_1$ and $w_- = E_2$.
After subtracting equations \eqref{systeps:E1} and \eqref{systeps:E2} and passing to the limit, we recover \eqref{eqap:w}. It remains to pass into the limit in the linear equations \eqref{systeps:F1} and \eqref{systeps:F2}.
\end{proof}

\section{Comparison principle}

In this section, we prove the comparison principle stated in Proposition \ref{prop:compar}.
  We use the Stampacchia method, in the spirit of \cite{Perthame}.
  Denoting $\delta w = \underline{w} - \overline{w}$, $\delta F_1 = \underline{F_1}-\overline{F_1}$, $\delta F_2 = \underline{F_2} - \overline{F_2}$, we have from \eqref{def:sub} and \eqref{def:super}
  \begin{align*}
    \partial_t (\delta w) =
    & b_1 (\delta F_1) \left(1-\frac{(\underline{w})_+}{K_1}\right) - \left(\frac{b_1 \overline{F_1}}{K_1}+\mu_1+\nu_1\right) ((\underline{w})_+-(\overline{w})_+) \\
    & - b_2 (\delta F_2) \left( 1-\frac{(\underline{w})_-}{K_2} \right) + \left( \frac{b_2 \overline{F_2} }{K_2} + \mu_2 + \nu_2 \right) ((\underline{w})_- - (\overline{w})_-).
  \end{align*}
  We have the obvious inequalities
  $$
  ((\underline{w})_+ - (\overline{w})_+) (\underline{w} - \overline{w})_+ \geq 0, \qquad ((\underline{w})_- - (\overline{w})_-) (\underline{w} - \overline{w})_+ \leq 0.
  $$
  Then, multiplying the latter equation with $(\delta w)_+$, we obtain
  \begin{align*}
    \frac 12 \partial_t (\delta w)_+^2 \leq\ 
    & b_1 (\delta F_1) (\delta w)_+ \left(1-\frac{(\underline{w})_+}{K_1}\right) - b_2 (\delta F_2) (\delta w)_+ \left( 1-\frac{(\underline{w})_-}{K_2} \right)  \\
    \leq\ & b_1 (\delta F_1)_+ (\delta w)_+ + b_2 (\delta F_2)_- (\delta w)_+ \\
    \leq\ & \frac{b_1+b_2}{2} (\delta w)_+^2 + \frac{b_1}{2} (\delta F_1)_+^2 + \frac{b_2}{2} (\delta F_2)_-^2.
  \end{align*}

  For the equation for $F_1$, we have
  \begin{align*}
    \frac 12 \partial_t (\delta F_1)_+^2 - D_1 \partial_{xx} (\delta F_1) (\delta F_1)_+
    & = \rho \nu_1 ((\underline{w})_+ - (\overline{w})_+) (\delta F_1)_+ - \delta_1 (\delta F_1)_+^2  \\
    & \leq \rho \nu_1 (\delta w)_+ (\delta F_1)_+  \leq \frac{\rho\nu_1}{2} ((\delta w)_+^2 + (\delta F_1)_+^2),
  \end{align*}
  where we used the inequality $(\underline{w})_+ - (\overline{w})_+ \leq (\underline{w}-\overline{w})_+$.
  Similarly, we have
  \begin{align*}
    \frac 12 \partial_t (\delta F_2)_-^2 - D_2 \partial_{xx} (\delta F_2) (\delta F_2)_-
    & = \rho \nu_2 ((\overline{w})_- - (\underline{w})_-) (\delta F_2)_- - \delta_2 (\delta F_2)_-^2  \\
    & \leq \rho \nu_2 (\delta w)_+ (\delta F_2)_-  \leq \frac{\rho\nu_2}{2} ((\delta w)_+^2 + (\delta F_2)_-^2),
  \end{align*}
  where we used the inequality $(\overline{w})_- - (\underline{w})_- \leq (\underline{w}-\overline{w})_+$.

  Finally, integrating with respect to $x\in\RR$, using an integration by parts, and adding the inequalities, we obtain
  \begin{align*}
    & \frac{d}{dt} \int_\RR ((\delta w)_+^2 + (\delta F_1)_+^2 + (\delta F_2)_-^2)\,dx +  \int_\RR (D_1 (\partial_x (\delta F_1)_+)^2 + D_2 (\partial_x (\delta F_2)_-)^2)\,dx \\
    & \qquad \leq (b_1+b_2+\rho(\nu_1+\nu_2)) \int_\RR (\delta w)_+^2 \,dx + (b_1+\rho\nu_1) \int_\RR (\delta F_1)_+^2\,dx + \int_\RR (\delta F_2)_-^2\,dx.
  \end{align*}
  We conclude thanks to a Gronwall argument.

\section{Technical Lemmas}\label{app:lem}
\subsection{Proof of Lemma \ref{estim_simpl}}
  We recall that since $F_1^*$ is a stationary solution, we have the relation $\frac{\rho \nu_1 b_1}{\frac{b_1}{K_1} F_1^* +\mu_1+\nu_1} = \delta_1$.
  Then, for all $x>0$, by definition of $\overline{F_1}$, we have
  \begin{align*}
    - D_1 \overline{F_1}'' & = \delta_1 (F_1^* -\overline{F_1}) = \frac{\rho \nu_1 b_1 F_1^*}{\frac{b_1}{K_1} F_1^* +\mu_1+\nu_1} - \delta_1 \overline{F_1} \\
    & \geq \frac{\rho \nu_1 b_1 \overline{F_1}}{\frac{b_1}{K_1} \overline{F_1} +\mu_1+\nu_1} - \delta_1 \overline{F_1},
  \end{align*}
where we have used the fact that $\overline{F_1}\leq F_1^*$ by definition.

For the second inequality, we compute
$$
-D_2 \underline{F_2} '' = -\delta_2 \underline{F_2}.
$$
Since the positive part is non-negative, the inequality follows straightforwardly.
\qed

\subsection{Proof of Lemma \ref{lem:estim2}}

  We first notice that since by definition $\overline{F_1}$ is increasing with $\overline{F_1}(0) = 0$ and $\underline{F_2}$ is decreasing with $\underline{F_2}(+\infty) = 0$, there exists a unique $\tilde{L}$ such that $b_1 \overline{F_1}(\tilde{L}) = b_2 \underline{F_2}(\tilde{L})$. Thus, the function $\overline{F_2}$ defined in \eqref{def:F2bar} is well-defined, and by straightforward computations, we have
  $$
  -D_2 \overline{F_2}'' = \delta_2 F_2^* \mathbf{1}_{(0,L)}(x) - \delta_2 \overline{F_2}, \qquad \text{ on } \RR^+.
  $$
  Then, $\overline{F_2}$ is a $C^1$ function, and we verify easily from its expression \eqref{def:F2bar} that it is decreasing on $(0,+\infty)$, with $\overline{F_2}(0) = F_2^*$ and $\overline{F_2}(+\infty) = 0$.
  Moreover, since $\underline{F_2}$ is a sub-solution for the latter equation, we deduce that $\underline{F_2} \leq \overline{F_2}$ on $\RR^+$.
  We deduce in particular
  $$
  b_1 \overline{F_1}(\tilde{L}) = b_2 \underline{F_2}(\tilde{L}) \leq b_2 \overline{F_2}(\tilde{L}).
  $$
  Then, $\overline{F_1}$ being increasing and $\underline{F_2}$ decreasing with limit $0$ at infinity, there exists a unique $L_0 > \tilde{L}$ such that
  $$
  b_1 \overline{F_1}(L_0) =  b_2 \overline{F_2}(L_0).
  $$
  Denoting $\chi_0 := \overline{F_1}(L_0) \in (0,F_1^*)$, this latter equality reads also
  $$
  b_1 \chi_0 = b_2 \overline{F_2}(\overline{F_1}^{-1}(\chi_0)) = b_2  \overline{F_2}\left(-\sqrt{\frac{D_1}{\delta_1}} \ln \left(1-\frac{\chi_0}{F_1^*}\right)\right).
  $$
  Using again the fact that $\overline{F_2}$ is decreasing, we deduce that for all $y\leq \chi_0$, the term in the positive part in the expression \eqref{eq:Gamma} is non-positive. Thus, for $0\leq y\leq \chi_0$, we have
  $$
  \Gamma_1'(y) = -\delta_1 y. 
  $$
  For $y>\chi_0$, we have $-\sqrt{\frac{D_1}{\delta_1}} \ln \left(1-\frac{y}{F_1^*}\right)\geq L_0 > \tilde{L}$. Hence, from \eqref{def:F2bar}, we get by deriving \eqref{eq:Gamma} for $y>\chi_0$,
  $$
  \Gamma_1'(y) = 
    \frac{\rho \nu_1}{\frac{b_1}{K_1} y + \mu_1 + \nu_1} \left(b_1 y - b_2 F_2^* \cosh\left(\sqrt{\frac{\delta_2}{D_2}}\tilde{L}\right) \left(1-\frac{y}{F_1^*}\right)^{\zeta}\right) - \delta_1 y,
  $$
  with $\zeta = \sqrt{\dfrac{\delta_2 D_1}{\delta_1 D_2}}$.  
  Then, recalling the relation  $\rho \nu_1 b_1 = \delta_1(\frac{b_1}{K_1} F_1^* +\mu_1+\nu_1)$, we get, for $y>\chi_0$,
  \begin{align*}
    \Gamma_1'(y) =
    & \frac{1}{\frac{b_1}{K_1} y + \mu_1 + \nu_1} \left(\left(\rho \nu_1b_1  - \delta_1(\frac{b_1}{K_1} y + \mu_1 + \nu_1)\right)y - \rho \nu_1 b_2 F_2^* \cosh\left(\sqrt{\frac{\delta_2}{D_2}}\tilde{L}\right) \left(1-\frac{y}{F_1^*}\right)^{\zeta}\right)  \\
    = & \frac{1}{\frac{b_1}{K_1} y + \mu_1 + \nu_1} \left( \frac{\delta_1 b_1}{K_1} y (F_1^*-y) - \rho \nu_1 b_2 F_2^* \cosh\left(\sqrt{\frac{\delta_2}{D_2}}\tilde{L}\right) \left(1-\frac{y}{F_1^*}\right)^{\zeta} \right)   \\
    = & \frac{ (F_1^*-y)}{\frac{b_1}{K_1} y + \mu_1 + \nu_1} \left( \frac{\delta_1 b_1}{K_1} y  - \rho \nu_1 b_2 \frac{F_2^*}{(F_1^*)^\zeta} \cosh\left(\sqrt{\frac{\delta_2}{D_2}}\tilde{L}\right) (F_1^*- y)^{\zeta-1} \right).
  \end{align*}
  From this latter expression, we may deduce the following facts~:
  \begin{itemize}
  \item If $\zeta \geq 1$, we verify easily that the right hand side vanishes only once on $(\chi_0,F_1^*)$ in a point denoted $y_0$. Then, $\Gamma_1$ is non-increasing on $(0,y_0)$ and non-decreasing on $(y_0,F_1^*)$. Hence $\max_{[0,F_1^*]} \Gamma_1 = \Gamma_1(F_1^*)$.
    
  \item If $\zeta<1$, either there is no root or there are two zeros for $\Gamma_1'$ on $(0,F_1^*)$. If there is no root, then $\Gamma_1$ is decreasing on $(0,F_1^*)$ with $\Gamma_1(0)=0$; it is a contradiction with \eqref{hypGamma}. Hence, $\Gamma_1'$ admits two zeros on $(0,F_1^*)$, there exists $\chi_0<y_1<\overline{U}<F_1^*$ such that $\Gamma_1$ is non-increasing on $(0,y_1)$, non-decreasing on $(y_1,\overline{U})$, and non-increasing on $(\overline{U},F_1^*)$, then $\max_{[0,F_1^*]} \Gamma_1 = \Gamma_1(\overline{U}) > 0$.
  \end{itemize}
  
  We observe that if $U$ is a solution of \eqref{eq:subsys}, then the quantity $\frac 12 D_1 (U')^2 + \Gamma_1(U)$ is conserved. Indeed the right hand side of the equation in \eqref{eq:subsys} may be rewritten $\Gamma_1'(U)$. Then, let us introduce the Cauchy problem
  $$
  U' = \sqrt{\frac{2}{D_1} \left(\max_{[0,F_1^*]} \Gamma_1 - \Gamma_1(U) \right)}, \qquad U(0) = 0.
  $$
  From assumption \eqref{hypGamma}, we clearly have $\max_{[0,F_1^*]} \Gamma_1>0$, hence the latter Cauchy problem is well-defined. There exists a solution which is non-decreasing and non-negative by construction. Moreover, this solution is bounded by $F_1^*$ if $\zeta>1$ or by $\overline{U}$ if $\zeta<1$. Hence it has a limit when $y\to +\infty$, this limit should be a stationary solution of the latter ODE, the only possibility is $\arg\max_{[0,F_1^*]} \Gamma_1$.
  With the above discussion, we have $\arg\max_{[0,F_1^*]} \Gamma_1 = F_1^*$ if $\zeta\geq 1$, whereas $\arg\max_{[0,F_1^*]} \Gamma_1 = \overline{U}$ if $\zeta<1$.
\qed

\subsection{Proof of Lemma \ref{lem:supsol}}
  The limits at $0$ and at infinity are clear by construction.
  On the one hand, we have already observed that by construction
  \begin{align*}
    & - D_2 \underline{F_2}'' = - \delta_2 \underline{F_2},  \\
    & - D_2 \overline{F_2}'' = \delta_2 F_2^* \mathbf{1}_{(0,\tilde{L})}(x) - \delta_2 \overline{F_2}.
  \end{align*}
  Clearly, $\underline{F_2}$ is a sub-solution for the equation satisfied by $\overline{F_2}$; it implies  $\underline{F_2} \leq \overline{F_2}$.
  Moreover, recalling the identity $\delta_2 = \frac{\rho \nu_2 b_2}{\frac{b_2}{K_2} F_2^* + \mu_2 + \nu_2}$, we get
  \begin{align*}
  -D_2 \overline{F_2}'' & = \delta_2 (F_2^* \mathbf{1}_{(0,\tilde{L})} - \overline{F_2}) = \frac{\rho \nu_2}{\frac{b_2}{K_2}F_2^* + \mu_2 + \nu_2} b_2 F_2^* \mathbf{1}_{(0,\tilde{L})}- \delta_2 \overline{F_2}   \\
  & \geq \frac{\rho \nu_2}{\frac{b_2}{K_2} \overline{F_2} +\mu_2+\nu_2}  b_2\overline{F_2} \mathbf{1}_{(0,\tilde{L})}(x) - \delta_2 \overline{F_2}  \\
  & \geq \frac{\rho \nu_2}{\frac{b_2}{K_2} \overline{F_2} +\mu_2+\nu_2} (b_2\overline{F_2} -b_1 \underline{F_1})_+ - \delta_2 \overline{F_2}.
\end{align*}

On the other hand, we have
$$
- D_1 \overline{F_1}'' = \delta_1 F_1^* - \delta_1 \overline{F_1}.
$$
And, from \eqref{eq:subsys} and the expression of $\Gamma_1$ in \eqref{eq:Gamma}, we have
\begin{align}
  \label{eqF_11}
  - D_1 \underline{F_1}''
  & = \frac{\rho \nu_1}{\frac{b_1}{K_1} \underline{F_1} + \mu_1 + \nu_1}\left( b_1 \underline{F_1} - b_2 \overline{F_2}\left(-\sqrt{\frac{D_1}{\delta_1}} \ln \left(1-\frac{\underline{F_1}}{F_1^*}\right)\right)\right)_+ - \delta_1 \underline{F_1}    \\
  \nonumber
  & \leq \frac{\rho \nu_1 b_1 \underline{F_1}}{\frac{b_1}{K_1} \underline{F_1} + \mu_1 + \nu_1} - \delta_1 \underline{F_1} \leq \frac{\rho \nu_1 b_1 F_1^*}{\frac{b_1}{K_1} F_1^* + \mu_1 + \nu_1} - \delta_1 \underline{F_1} = \delta_1 F_1^* - \delta_1 \underline{F_1}.
\end{align}
Hence, we have $\underline{F_1}\leq \overline{F_1}$. It implies for all $y>0$,
$$
-\sqrt{\frac{D_1}{\delta_1}} \ln \left(1-\frac{\underline{F_1}(y)}{F_1^*}\right) \leq
-\sqrt{\frac{D_1}{\delta_1}} \ln \left(1-\frac{\overline{F_1}(y)}{F_1^*}\right) = y,
$$
where we use the definition of $\overline{F_1}$ given in Lemma \ref{estim_simpl} for the last identity.
Thus, injecting this latter inequality into \eqref{eqF_11} we have obtained, since $\overline{F_2}$ is non-increasing, for all $y>0$,
$$
  - D_1 \underline{F_1}''(y)
  \leq \frac{\rho \nu_1}{\frac{b_1}{K_1} \underline{F_1}(y) + \mu_1 + \nu_1}\left( b_1 \underline{F_1}(y) - b_2 \overline{F_2}(y)\right)_+ - \delta_1 \underline{F_1}(y).
$$

\section*{Acknowledgments}

The author acknowledges partial support from the STIC AmSud project BIO-CIVIP 23-STIC-02.

\end{document}